\documentclass{amsart}
\usepackage{cite}
\usepackage{graphicx}
\usepackage{placeins}
\usepackage{caption}
\usepackage{amssymb,amsmath,amsfonts}
\newtheorem{theorem}{Theorem}[section]
\newtheorem{lemma}[theorem]{Lemma}
\newtheorem{cor}[theorem]{Corollary}

\newtheorem{definition}[theorem]{Definition}

\newtheorem{example}[theorem]{Example}

\newtheorem{remark}[theorem]{Remark}

\numberwithin{equation}{section}
\begin{document}
\title[fixed point problems]
{ The Problem of Split Equality Fixed-Point and its Applications   }
\date{}
\author[L.B. Mohammed, et al ] { $^{1}$Lawan Bulama Mohammed, $^{2}$Adem Kilicman }
\address{$^{1}$Department of Mathematics,  Faculty of Science, 
	\newline \indent \hspace{2mm}Federal  University Dutse, PMB 7156, Dutse, Jigawa State, Nigeria}
\email{ $^{1}$lawanbulama@gmail.com}
 \address{$^{2}$ Institute for Mathematics Research, 
 	\newline \indent \hspace{2mm}Universiti Putra Malaysia, 43400 Serdang, Selangor, Malaysia}
 \email{ akilic@upm.edu.my}

\begin{abstract} 
It is generally known that in order to solve the split equality fixed-point problem (SEFPP), it is necessary to compute the norm of bounded and linear operators, which is a challenging task in real life, to address this issue, we studied the SEFPP involving the class of quasi-pseudocontractive mappings in Hilbert spaces and constructed novel algorithms in this regards, and  we proved the algorithms' convergence both with and without prior knowledge of the operator  norm for bounded and linear mappings. Additionally, we gave applications and numerical examples of our findings. A variety of well-known discoveries revealed in the literature are generalized by the findings presented in this work.
\end{abstract} 
\maketitle
\section{Introduction}
  In this manuscript, we consistently utilize the following notations: $\left\langle .,.\right\rangle$ denotes an inner product, and $\left\|.\right\|$ represents its corresponding norm. We designate $\mathbb{H}{j}, j=1,2,3,$ as Hilbert spaces, $\mathbb{K}{j}$ as nonempty, convex, and closed subsets of $\mathbb{H}{j},$ and $\mathbb{D}{j}: \mathbb{H}{j}\to\mathbb{ H}{j}, j=1,2,$ as bounded and linear mappings. The symbols $\rightharpoonup$ and $\to$ signify weak and strong convergences, respectively.

 A mapping $\mathbb{T}:\mathbb{H}_{1}\to \mathbb{H}_{1}$ is known as  a fixed point of $\mathbb{T}$ \big(Fix($\mathbb {T}$)\big) if $\mathbb{T}p=p.$ We denote the set of   Fix($\mathbb{T}$) by  $\{p\in Fix(\mathbb{T}): \mathbb{T}p=p\}$. $\mathbb{T}$ is known to be	quasi nonexpansive, if 
 $\|\mathbb{T}y-p\|\leq \|y-p\|, \forall y\in \mathbb{H}_{1} \rm{~and~} p\in Fix(\mathbb{T}).$ It is obvious that if $\mathbb{T}$ is quasi nonexpansive, then $ \left\|\mathbb{T}y-y\right\|\leq2\left\langle y-\mathbb{T}y, y-p\right\rangle.$ 
 $\mathbb{T}$ is called  demicontractive, if 
 	$\|p-\mathbb{T}y\|^{2} \leq$ $\|p-y\|^{2}$ $+\beta\|y-\mathbb{T}y\|^{2},$ $\forall y\in \mathbb{H}_{1},  p\in Fix(\mathbb{T}) \rm{~and~}\beta\in [0, 1),$  and it is called quasi-pseudocontractive  if $\beta=1.$ 
 $\mathbb{T}$ is known as $\beta-$strongly monotone if
 	$\left\langle \mathbb{T}y-\mathbb{T}p, y-p\right\rangle\geq \beta \|\mathbb{T}y-\mathbb{T}p\|^{2},\forall y,p\in \mathbb{H}_{1}, \beta > 0.$

\begin{remark}
	The   quasi-pseudocontractive mappings encompasses various types, including quasi-nonexpansive, demicontractive, and several others. For further details, please refer to \cite{a} and the cited references therein.
\end{remark}

The problem of finding 
\begin{equation}\label{eqn1}
	p^{*}\in\mathbb{K}_{1} {\rm~such~ that~}  \mathbb{D}_{1}p^{*}\in \mathbb{K}_{2}
	\end{equation} 
is known as  "Split Feasibility Problem (SFP)". The SFP, initially introduced by \cite{1}, has garnered significant attention from researchers due to its versatile applications in practical domains. These applications span a wide spectrum, encompassing fields such as signal processing, intensity-modulated radiation therapy, and image reconstruction, as documented in \cite{5,6,7}.

If the problem stated in (\ref{eqn1}) possesses a solution, it is evident that $p^{*}\in\mathbb{K}_{1}$ is a solution to equation (\ref{eqn1}) if and only if it also solves

\begin{equation}p^{*}=P_{\mathbb{K}_{1}}\Big(I-\lambda\mathbb{D}_{1}^{*}
(I-P_{\mathbb{K}_{2}})\mathbb{D}_{1}\Big)p^{*}, \forall	p^{*}\in\mathbb{K}_{1},
\end{equation}
where $\lambda >0,$ $P_{\mathbb{K}_{j}}$ are metric projection on  ${K}_{j}, j=1,2,$ respectively.
To address the problem (\ref{eqn1}), Byrne \cite{10} proposed the CQ-algorithm through the following introduced algorithm:

\begin{equation}\label{b1}
	x_{n+1}=P_{\mathbb{K}_{1}}\Big(I-\lambda\mathbb{D}_{1}^
	{*}(I-P_{\mathbb{K}_{2}})\mathbb{D}_{1}\Big)x_{n},  
\end{equation}
where $\lambda \in (0, \frac{2}{\mathbb{D}_{1}\mathbb{D}_{1}^{*}}).$  Algorithm (\ref{b1}) involves the computation of $P_{\mathbb{K}_{j}}$ onto  ${K}_{j}, j=1,2,$ and  this is known to be implementable when the projections  
have closed-form expressions.

Related to SFP, we have the "Split Equality Problem (SEP)." This problem was introduced by Moudafi and  Al-Shemas \cite{moudafi} and it entails as finding
\begin{equation}
	~ p^{*}\in \mathbb{K}_{1}~{\rm  and~} q^{*}\in \mathbb{K}_{2}~ ~{\rm ~such~ that~} \mathbb{D}_{1}p^{*}=\mathbb{D}_{2}q^{*}.\label{eqn21}
\end{equation}
By using $\mathbb{D}_{2}=I$ (identity mapping), it is easy to see that SEP reduces to SFP. The following algorithm was taken into consideration by  Moudafi and  Al-Shemas\cite{moudafi}  to solve SEP: 
\begin{equation}
	{} \left\{\begin{array}{ll}x_{n+1}=P_{\mathbb{K}_{1}}\Big(x_{n}-\lambda_{n}\mathbb{D}_{1}^{*}(\mathbb{D}_{1}x_{n}-\mathbb{D}_{2}y_{n})\Big), n\geq 0;
		\\ y_{n+1}=P_{\mathbb{K}_{2}}\Big(y_{n}
		+\lambda_{n}\mathbb{D}_{2}^{*}(\mathbb{D}_{1}x_{n}
		-\mathbb{D}_{2}y_{n})\Big), n\geq 0;\label{eqn3we} 
	\end{array}\right.
\end{equation} 
where         $(x_{0},y_{0})\in \mathbb{H}_{1}\times \mathbb{H}_{2}$ were chosen arbitrary, $P_{\mathbb{K}_{j}}, j=1,2,$ are metric projections on $K_{j}s.$   After some certain condition imposed on $\lambda_{n},$  Moudafi and  Al-Shemas \cite{moudafi}  obtained a weak convergence result.

As any nonempty, closed, and convex subset of a Hilbert space can be viewed as the fixed point set of its corresponding projector, equation (\ref{eqn21}) is consequently simplified to 
\begin{equation}
	\rm{find}	~ p^{*}\in Fix(\mathbb{T}_{1})~{\rm and~} q^{*}\in Fix(\mathbb{T}_{1})~ \ni \mathbb{D}_{1}p^{*}=\mathbb{D}_{2}q^{*},\label{eqn31}
\end{equation}
where $\mathbb{T}_{j}:\mathbb{K}_{j}\to H_{j},j=1,2,$ are  nonlinear operators with $Fix(\mathbb{T}_{j})\neq\emptyset.$  Equation (\ref{eqn31}) is regarded as "Split Equality Fixed Point Problem (SEFPP)
".

Motivated by the result in \cite{moudafi},   Moudafi  \cite{mou}  considered the algorithm below:
\begin{equation}
	{} \left\{\begin{array}{ll}x_{n+1}=\mathbb{T}_{1}\Big(x_{n}-\lambda_{n}\mathbb{D}_{1}^{*}(\mathbb{D}_{1}x_{n}-\mathbb{D}_{2}y_{n})\Big), n\geq 0;
		\\ y_{n+1}=\mathbb{T}_{2}\Big(y_{n}
		+\lambda_{n}\mathbb{D}_{2}^{*}(\mathbb{D}_{1}x_{n}
		-\mathbb{D}_{2}y_{n})\Big), n\geq 0;\label{a1} 
	\end{array}\right.
\end{equation}
after some   conditions imposed on the parameters and operators involved, they proved  weak convergence results.

To solve (\ref{a1}), the  inverse of   bounded and linear operator must be computed and this is known to be none easy task, this was why  Byne  \cite{by} considered an algorithm for solving     SFP without the said inverse.  

It was reported in \cite{b} that the SFP can be reduced to convex feasibility problem (CFP) as 
well as fixed point problem (FPP).  Fixed point theory (FPT) can be     considered
a core area of research of nonlinear analysis, this is due to its various   applications  
in many area of researches such as in image processing, equilibrium
problems,  the study of existence and uniqueness of solutions for the integral and
differential equations,  selection and matching problems see \cite{LBM}.

 Moudafi's algorithm in \cite{mou} involved the firmly quasi-nonexpansive mappings;
this mapping  includes  the class of quasi-nonexpansive. Very recently, Che and Li \cite{LI}, considered the following  algorithm for finding the solution of split equality problems (SEP) and proved the convergence results of the algorithm:

\begin{equation}
	{} \left\{\begin{array}{ll}x_{n+1}=\beta_{n}u_{n}+(1-\beta_{n})\mathbb{T}_{1}u_{n};
		\\u_{n}=x_{n}-\lambda_{n}\mathbb{D}_{1}^{*}(\mathbb{D}_{1}x_{n}-\mathbb{D}_{2}y_{n}, n\geq 0;\\

		\\ y_{n+1}=\beta_{n}v_{n}+(1-\beta_{n})\mathbb{T}_{2}v_{n};
		\\v_{n}= y_{n}
		+\lambda_{n}\mathbb{D}_{2}^{*}(\mathbb{D}_{1}x_{n}
		-\mathbb{D}_{2}y_{n}), n\geq 0;\label{a111} 
	\end{array}\right.
\end{equation}
where $\mathbb{T}_{1}$ and $\mathbb{T}_{2}$ are strictly pseudononspreading mappings.

Since  quasi-pseudocontractive mapping includes;  firmly quasi-nonexpansive, quasi nonexansive, nonexansive, directed  and demicontractive mappings, this motivated Chang et al.\cite{a} to introduce  the following   algorithm  for solving the SEFPP involving quasi-pseudocontractive mappings and proved the weak  convergence result of the   algorithm:

\begin{equation}
	{} \left\{\begin{array}{ll}x_{n+1}=\beta_{n}x_{n}+(1-\beta_{n})\Big((1-\eta)I +\eta \mathbb{T}_{1}((1-\zeta)I+\zeta \mathbb{T}_{1})+\beta_{n}\Big)u_{n};
		\\u_{n}=x_{n}-\lambda_{n}\mathbb{D}_{1}^{*}(\mathbb{D}_{1}x_{n}-\mathbb{D}_{2}y_{n}), n\geq 0.\\

		\\ y_{n+1}=\beta_{n}x_{n}+(1-\beta_{n})\Big((1-\eta)I +\eta \mathbb{T}_{1}((1-\zeta)I+\zeta \mathbb{T}_{2})+\beta_{n}\Big)u_{n};
		\\v_{n}= y_{n}
		+\lambda_{n}\mathbb{D}_{2}^{*}(\mathbb{D}_{1}x_{n}
		-\mathbb{D}_{2}y_{n}), n\geq 0.\label{a111} 
	\end{array}\right.
\end{equation}

Chang et al.'s algorithm \cite{a} necessitates prior operator norm understanding. Boikanyo and Zegeye \cite{z}, building on Chang et al.'s findings, established strong convergence of SEFPP without operator norm prerequisites or semi-compactness conditions.  Similar result had been also being proved in \cite{2}. Very recently, Mohammed et al.  \cite{LBM} proved the convergence of the following algorithm for the class of total quasi-asymptotically nonexpansive mappings:

\begin{equation}
	{} \left\{\begin{array}{ll}x_{n+1}=(1-\beta_{n})u_{n}+\beta_{n}\mathbb{T}^{n}_{1}u_{n};
		\\u_{n}=(1-\gamma_{n})x_{n}+\gamma_{n}\mathbb{T}^{n}_{1}x_{n}-\lambda_{n}\mathbb{D}_{1}^{*}(\mathbb{D}_{1}x_{n}-\mathbb{D}_{2}y_{n}), n\geq 0;\\

		\\ y_{n+1}=(1-\beta_{n})v_{n}+\beta_{n}\mathbb{T}^{n}_{2}v_{n};
		\\v_{n}=(1-\gamma_{n})y_{n}+\gamma_{n}\mathbb{T}^{n}_{1}y _{n}-\lambda_{n}\mathbb{D}_{2}^{*}(\mathbb{D}_{1}x_{n}
		-\mathbb{D}_{2}y_{n}), n\geq 0;\label{a111} 
	\end{array}\right.
\end{equation}

 The findings in \cite{a, 2,LBM}  both converged weakly, in an infinite dimensional space, weak convergent does not imply strong convergent, whereas   strong and weak convergences   coincide if the dimension is finite.
 Based on these findings, we set out to develop new methods for solving SEFPP for quasi-pseudocontractive mappings in Hilbert spaces and to demonstrate how the suggested methods converge. The suggested algorithms' convergence findings will also be presented in a form that frees them from the constraints imposed by the operator norm of bounded and linear operators. At the end, We provided   numerical examples that highlight our findings. 
 
 The following is how the paper is set up: The background of the study is explained in the introduction, some fundamental findings are presented in the preliminary section. The main result was discussed in section three while in section four applications and numerical results were being discussed.

\section{Preliminaries}
 This section offers a few fundamental findings that help support the paper's primary findings.
 \begin{definition}\label{lemma01}
 	A mapping $\mathbb{T}: \mathbb{H}_{1}\to \mathbb{H}_{1}$ is said to semi-compact if for any bounded sequence $\{x_{n}\}\subseteq \mathbb{H}_{1}$ with $\|x_{n}-\mathbb{T}x_{n}\|\to 0,$ then there exist $\{x_{n_{i}}\}\subseteq \{x_{n}\}$ such that $x_{n_{i}}\to x\in \mathbb{H}_{1}.$ \end{definition}

\begin{lemma}\label{lemma1}\cite{2} Suppose $\mathbb{T}: \mathbb{H}_{1}\to \mathbb{H}_{1}$ is Lipschitz  with $L>0,$ and  $\mathbb{U}:= (1-\eta)I +\eta \mathbb{T}((1-\zeta)I+\zeta \mathbb{T}),$ then
\begin{enumerate}
\item[a.] $Fix(\mathbb{T})=Fix((1-\eta)I +\eta \mathbb{T}((1-\zeta)I+\zeta \mathbb{T})=Fix(\mathbb{U});$
\item[b.] $\mathbb{U}$ is demiclosed at zero only if $\mathbb{T}$ is demiclosed at zero.
\item[c.] $\mathbb{U}$ is Lipschitz with $L^{2}$
\item[d.] $\mathbb{U}$ is quasi-nonexpansive only if $\mathbb{T}$ is quasi-pseudocontractive.
\end{enumerate}
\end{lemma}

\begin{lemma}\label{opial}{\cite{3}} Suppose $\Omega\subseteq\mathbb{H},$ for $\{x_{n}\}\subseteq \Omega,$ then 
\begin{enumerate}
\item[a.]  $\underset{n\to\infty}{\lim}{\|x_{n}-x\|},$ $\forall x\in\Omega$, exist;
\item[b.] For any weak-cluster point of $x_{n}\subseteq \Omega,$   then   $x_{n} \rightharpoonup y,$ $y\in\Omega.$ 
\end{enumerate}
\end{lemma}

\begin{lemma}\label{lemma3}\cite{21} Let  $\{x_{n}\}, \{\gamma_{n}\}\subseteq \mathbb{R}^{+} $ such  that
	$$x_{n+1}\leq (1-\xi_{n})x_{n}+\gamma_{n}, \forall n\geq 0,$$
where $\{\xi_{n}\}\subseteq(0,1),$ then
 \begin{enumerate}
\item[a.]$\underset{n\to\infty}{\lim}\xi_{n}=0 {\rm ~and~} \underset{n}\sum\xi_{n}=\infty;$
\item[b.]$\underset{n\to\infty}{\limsup}\frac{\gamma_{n}}{\xi_{n}}\leq 0 {\rm ~or~} \underset{n}\sum |\gamma_{n}|<\infty,$
then, $\underset{n\to\infty}{\lim}x_{n}=0.$
\end{enumerate}
\end{lemma}

\begin{lemma}\label{lemmat}\cite{21} Let  $\{\xi_{n}\},$  $\{\eta_{n}\}\subseteq \mathbb{R}^{+}$  such that $\sum_{n=0}^{\infty}\eta_{n}<\infty.$
If $$\xi_{n+1}\leq (1+\eta_{n})\xi_{n}~\rm{or~} \xi_{n+1}\leq \xi_{n}+\eta_{n}, \forall n\geq 0,$$ then $\underset{n\to\infty}{\lim}\xi_{n}$ exist.
\end{lemma}

\section{Main Results}
\subsection{SEFPP With   Prior   Knowledge  on  Operator Norms}
The sequel will use $\Gamma$ to represent the solution set for (\ref{eqn31}), that is  \begin{equation}
 \Gamma:=\{{\rm  x^{*}\in Fix(\mathbb{T}_{1})~ and~ y^{*}\in Fix(\mathbb{T}_{2})~ such~ that~ \mathbb{D}_{1}x^{*}=\mathbb{D}_{2}y^{*}.}\}\label{1}
\end{equation}  
These presumptions were used to approximation  (\ref{1}). Suppose that

\begin{enumerate}
 	\item [(A1)] $\mathbb{T}_{j}:\mathbb{H}_{j}\to \mathbb{H}_{j},j=1,2,$ are two quasi-pseudocontractive operators with  $Fix(\mathbb{T}_{j})\neq\emptyset,j=1,2,$ in addition, suppose $\mathbb{T}_{1}$ is L- Lipschitz.
 	\item [(A2)]  $\mathbb{D}_{j}:\mathbb{H}_{j}\to \mathbb{H}_{j}, j=1,2$  are   linear and bounded operators  with their adjoints $\mathbb{D}_{1}^{*}$ and $\mathbb{D}_{2}^{*},$ respectively.
 	\item [(A3)] $(\mathbb{T}_{j}-I), j=1,2$  are demiclosed at origin.
 	\item [(A4)] Let $\mathbb{U}$ and $\mathbb{V}$ be defined below: 
 	\begin{equation}
 	{} \left\{\begin{array}{ll} \mathbb{U}=(1-\eta)I +\eta \mathbb{T}_{1}((1-\zeta)I+\zeta \mathbb{T}_{1}),
 	\\\mathbb{V}=(1-\eta)I +\eta \mathbb{T}_{2}((1-\zeta)I+\zeta \mathbb{T}_{2}), & \textrm{ $  $} \label{3} 
 	\end{array}\right.
 	\end{equation}
 	where $0<\eta<\zeta<\frac{1}{1+\sqrt{1+L^{2}}}.$
 	\item [(A5)] \textbf{Algorithm:} Let ($x_{n}, y_{n} )\subseteq \mathbb{H}_{1}\times \mathbb{H}_{2},$  be defined by  
 	\begin{equation}
 	{} \left\{\begin{array}{ll} x_{n+1}=(1-\alpha_{n}) v_{n}+\alpha_{n}\mathbb{U}v_{n};
 		\\v_{n}=(1-\tau_{n})x_{n}+\tau_{n}\mathbb{U} x_{n}+\tau_{n}\mathbb{D}_{1}^{*}(\mathbb{D}_{2}y_{n}-\mathbb{D}_{1}x_{n}), \forall n\geq 0.\\
 	\\ y_{n+1}=(1-\alpha_{n}) w_{n}+\alpha_{n}\mathbb{V}w_{n};
 	\\w_{n}=(1-\tau_{n}) y_{n}+\tau_{n}\mathbb{V} y_{n}+\tau_{n}\mathbb{D}_{2}^{*}(\mathbb{D}_{1}x_{n}-\mathbb{D}_{2}y_{n}), \forall n\geq 0.& \textrm{ $  $} \label{al1} 
 	\end{array}\right.
 	\end{equation}
 where $(x_{0}, y_{0})\in \mathbb{H}_{1}\times \mathbb{H}_{2}$ are chosen arbitrary,	$0<a<\alpha_{n}<1,$ and $\tau_{n}\in\left(0, \frac{2}{L_{1}+L_{2}}\right),$ where   $L_{1}=\mathbb{D}_{1}^{*}\mathbb{D}_{1}$ and $L_{2}=\mathbb{D}_{2}^{*}\mathbb{D}_{2},$ {\rm respectively~},  $\underset{n\to\infty}{\lim}\tau_{n}=0 {\rm ~and~}  \underset{n}\sum\tau_{n}=\infty.$ 
 
  \end{enumerate} 
 
 \begin{theorem}\label{T1} Suppose $(A1)-(A5)$ 
 	are held, and that  $\Gamma\neq\emptyset.$  
 	Then  $\{(x_{n}, y_{n})\}$ 
 	defined by (\ref{al1}) 
 	converges  to  $\Gamma,$ in addition if $\mathbb{T}_{1}$ and $\mathbb{T}_{2}$ are   semi-compact, then  $\{(x_{n}, y_{n})\}$  converges strongly to  $\Gamma.$
 \end{theorem}

\begin{proof}
Let $(p, q)\in\Gamma,$ noticing that $\mathbb{U}$ is quasi-nonexpansive (see, Lemma \ref{lemma1}),  then  by  (\ref{al1}), we have 
\begin{align}
\left\|x_{n+1}-p\right\|^{2}=&\left\|(1-\alpha_{n})v_{n}+\alpha_{n}\mathbb{U}v_{n}-p\right\|^{2}\nonumber
\\=&\left\|(1-\alpha_{n}) (v_{n}-p)+\alpha_{n}(\mathbb{U}v_{n}-p)\right\|^{2}\nonumber
\\=&(1-\alpha_{n})\left\|v_{n}-p\right\|^{2}+\alpha_{n}\left\|\mathbb{U}v_{n}-p\right\|^{2}\nonumber
-(1-\alpha_{n})\alpha_{n}\left\|\mathbb{U}v_{n}-v_{n}\right\|^{2}
\\\leq&    \left\|v_{n}-p\right\|^{2} {\rm ~and~} \label{A}
\\\left\|v_{n}-p\right\|^{2}=&\left\|(1-\tau_{n}) x_{n}+\tau_{n}\mathbb{U}x_{n}+\tau_{n}\mathbb{D}_{1}^{*}(\mathbb{D}_{2}y_{n}-\mathbb{D}_{1}x_{n})-p\right\|^{2}\nonumber
\\=&\left\|(1-\tau_{n})(x_{n}-p) +\tau_{n}(\mathbb{U}x_{n}-P)+\tau_{n}\mathbb{D}_{1}^{*}(\mathbb{D}_{2}y_{n}-\mathbb{D}_{1}x_{n})\right\|^{2}\nonumber
\\=&\left\|(1-\tau_{n})(x_{n}-p)+\tau_{n}(\mathbb{U}x_{n}-p)\right\|^{2}+\tau_{n}^{2}\left\|\mathbb{D}_{1}^{*}(\mathbb{D}_{2}y_{n}-\mathbb{D}_{1}x_{n})\right\|^{2}\nonumber
\\+& 2\tau_{n}\left\langle (1-\tau_{n})( x_{n}-p)+\tau_{n}(\mathbb{U}x_{n}-p), \mathbb{D}_{1}^{*}(\mathbb{D}_{2}y_{n}-\mathbb{D}_{1}x_{n})\right\rangle\nonumber
\\=&(1-\tau_{n})\left\|x_{n}-p\right\|^{2}+\tau_{n}\left\|\mathbb{U}x_{n}-p\right\|^{2}-(1-\tau_{n})\tau_{n}\left\|\mathbb{U}x_{n}-x_{n}\right\|^{2}\nonumber
\\+&\tau_{n}^{2}\left\|\mathbb{D}_{1}^{*}(\mathbb{D}_{2}y_{n}-\mathbb{D}_{1}x_{n})\right\|^{2}+2\tau_{n}(1-\tau_{n})\left\langle  x_{n}-p, \mathbb{D}_{1}^{*}(\mathbb{D}_{2}y_{n}-\mathbb{D}_{1}x_{n})\right\rangle\nonumber
\\+& 2\tau_{n}^{2}\left\langle  \mathbb{U}x_{n}-x_{n}+x_{n}-p, \mathbb{D}_{1}^{*}(\mathbb{D}_{2}y_{n}-\mathbb{D}_{1}x_{n})\right\rangle\nonumber
\\\leq&\left\|x_{n}-p\right\|^{2}-(1-\tau_{n})\tau_{n}\left\|\mathbb{U}x_{n}-x_{n}\right\|^{2}\nonumber
\\+&\tau_{n}^{2}\left\|\mathbb{D}_{1}^{*}(\mathbb{D}_{2}y_{n}-\mathbb{D}_{1}x_{n})\right\|^{2}
+2\tau_{n}(1-\tau_{n})\left\langle  x_{n}-p, \mathbb{D}_{1}^{*}(\mathbb{D}_{2}y_{n}-\mathbb{D}_{1}x_{n})\right\rangle\nonumber
\\+& 2\tau_{n}^{2}\left\langle  \mathbb{U}x_{n}-x_{n}, \mathbb{D}_{1}^{*}(\mathbb{D}_{2}y_{n}-\mathbb{D}_{1}x_{n})\right\rangle+ 2\tau_{n}^{2}\left\langle   x_{n}-p, \mathbb{D}_{1}^{*}(\mathbb{D}_{2}y_{n}-\mathbb{D}_{1}x_{n})\right\rangle\nonumber
\nonumber
\\\leq&\left\|x_{n}-p\right\|^{2}+\tau_{n}^{2}\left\|\mathbb{D}_{1}^{*}(\mathbb{D}_{2}y_{n}-\mathbb{D}_{1}x_{n})\right\|^{2}-(1-\tau_{n})\tau_{n}\left\|\mathbb{U}x_{n}-x_{n}\right\|^{2}\nonumber
\\+&2\tau_{n}\left\langle  \mathbb{D}_{1}x_{n}-\mathbb{D}_{1}p, \mathbb{D}_{2}y_{n}-\mathbb{D}_{1}x_{n}\right\rangle+2\tau_{n}^{2}\left\langle  \mathbb{U}\mathbb{D}_{1}x_{n}-\mathbb{D}_{1}x_{n}, \mathbb{D}_{2}y_{n}-\mathbb{D}_{1}x_{n}\right\rangle\nonumber
\\\leq & \left\|x_{n}-p\right\|^{2}
+\tau_{n}^{2}\left\|\mathbb{D}_{1}^{*}(\mathbb{D}_{2}y_{n}-\mathbb{D}_{1}x_{n})\right\|^{2}-(1-\tau_{n})\tau_{n}\left\|\mathbb{U}x_{n}-x_{n}\right\|^{2}\nonumber
\\+&2\tau_{n}\left\langle  \mathbb{D}_{1}x_{n}-\mathbb{D}_{1}p, \mathbb{D}_{2}y_{n}-\mathbb{D}_{1}x_{n}\right\rangle-\tau_{n}^{2}\left\|\mathbb{U}\mathbb{D}_{1}x_{n}-\mathbb{D}_{1}x_{n}\right\|^{2}.\label{B} 
\end{align}

By  (\ref{A}) and (\ref{B}), we have 
\begin{align}
\left\|x_{n+1}-p\right\|^{2}\leq&\left\|x_{n}-p\right\|^{2}+\tau_{n}^{2}L_{1}\left\|\mathbb{D}_{2}y_{n}-\mathbb{D}_{1}x_{n}\right\|^{2}-(1-\tau_{n})\tau_{n}\left\|\mathbb{U}x_{n}-x_{n}\right\|^{2}\nonumber
\\-&2\tau_{n}\left\langle  \mathbb{D}_{1}x_{n}-\mathbb{D}_{1}p, \mathbb{D}_{2}y_{n}-\mathbb{D}_{1}x_{n}\right\rangle-\tau_{n}^{2}\left\|\mathbb{U}\mathbb{D}_{1}x_{n}-\mathbb{D}_{1}x_{n}\right\|^{2}. \label{C1} 
\end{align}

Similarly, 
\begin{align}
\left\|y_{n+1}-q\right\|^{2}\leq&\left\|y_{n}-q\right\|^{2}+\tau_{n}^{2}L_{2}\left\|\mathbb{D}_{1}x_{n}-\mathbb{D}_{2}y_{n}\right\|^{2}-(1-\tau_{n})\tau_{n}\left\|\mathbb{V}y_{n}-y_{n}\right\|^{2}\nonumber
\\+&2\tau_{n}
\left\langle  \mathbb{D}_{2}y_{n}-\mathbb{D}_{2}q, \mathbb{D}_{1}x_{n}-\mathbb{D}_{2}y_{n}\right\rangle-\tau_{n}^{2}\left\|\mathbb{V}\mathbb{D}_{2}y_{n}-\mathbb{D}_{2}y_{n}\right\|^{2}. \label{D1} 
\end{align}

 Since  $\mathbb{D}_{1}p=\mathbb{D}_{2}q,$ and coupled with (\ref{C1}) and (\ref{D1}), we have
\begin{align}
\left\|x_{n+1}-p\right\|^{2}+\left\|y_{n+1}-q\right\|^{2}
\leq&\left\|x_{n}-p\right\|^{2}+\left\|y_{n}-q\right\|^{2}\nonumber
\\-&\tau_{n}\big(2-\tau_{n}(L_{1}+L_{2})\big)\left\|\mathbb{D}_{1}x_{n}-\mathbb{D}_{2}y_{n}\right\|^{2}\nonumber
\\-&(1-\tau_{n})\tau_{n}\Big(\left\|\mathbb{U}x_{n}-x_{n}\right\|^{2}
+\left\|\mathbb{V}y_{n}-y_{n}\right\|^{2}\Big)\nonumber
\\-&\tau_{n}^{2}\Big(\left\|\mathbb{U}\mathbb{D}_{1}x_{n}-\mathbb{D}_{1}x_{n}\right\|^{2} +\left\|\mathbb{V}\mathbb{D}_{2}y_{n}-\mathbb{D}_{2}y_{n}\right\|^{2}\Big). \label{C} 
\end{align}

This implies that
\begin{align}
\Gamma_{n+1}\leq&\Gamma_{n}-\tau_{n}\big(2-\tau_{n}(L_{1}+L_{2})\big)\left\|\mathbb{D}_{1}x_{n}-\mathbb{D}_{2}y_{n}\right\|^{2}\nonumber
\\-&(1-\tau_{n})\tau_{n}\Big(\left\|\mathbb{U}x_{n}-x_{n}\right\|^{2}+\left\|\mathbb{V}y_{n}-y_{n}\right\|^{2}\Big)\nonumber
\\-&\tau_{n}^{2}\Big(\left\|\mathbb{U}\mathbb{D}_{1}x_{n}-\mathbb{D}_{1}x_{n}\right\|^{2} +\left\|\mathbb{V}\mathbb{D}_{2}y_{n}-\mathbb{D}_{2}y_{n}\right\|^{2}\Big),
\label{D2} 
\end{align}
~\rm{where}~$\Gamma_{n}:=\left\|x_{n}-p\right\|^{2}+\left\|y_{n}-q\right\|^{2}.$ 
Thus, $\Gamma_{n+1}\leq\Gamma_{n},$ therefore, $\Gamma_{n}$ is  decreasing sequence which  bounded from below by 0, hence, $\Gamma_{n}$ converges, therefore, we have that

\begin{align}
 \left\|\mathbb{D}_{1}x_{n}-\mathbb{D}_{2}y_{n}\right\|^{2}&\leq \frac{\Gamma_{n}-\Gamma_{n+1}}{\tau_{n}\big(2-\tau_{n}(L_{1}+L_{2})\big)}\nonumber
 \\&\leq \frac{2(\Gamma_{n}-\Gamma_{n+1})}{\big(2-\tau_{n}(L_{1}+L_{2})\big)(L_{1}+L_{2})},\nonumber
\end{align}

\begin{align}
	\left\|\mathbb{U}x_{n}-x_{n}\right\|^{2}+\left\|\mathbb{V}y_{n}-y_{n}\right\|^{2}&\leq \frac{\Gamma_{n}-\Gamma_{n+1}}{\tau_{n}\big(1-\tau_{n}\big)}\nonumber
	\\&\leq \frac{2(\Gamma_{n}-\Gamma_{n+1})}{\big(1-\tau_{n}\big)(L_{1}+L_{2}) }\nonumber {\rm~and~}
\end{align}
\begin{align}
\left\|\mathbb{U}\mathbb{D}_{1}x_{n}-\mathbb{D}_{1}x_{n}\right\|^{2} +\left\|\mathbb{V}\mathbb{D}_{2}y_{n}-\mathbb{D}_{2}y_{n}\right\|^{2}&\leq \frac{\Gamma_{n}-\Gamma_{n+1}}{\tau_{n}^{2}}\nonumber
	\\&\leq \frac{2(\Gamma_{n}-\Gamma_{n+1})}{L_{1}+L_{2} }.\nonumber 
\end{align}
These lead to


\begin{align}
\underset{n\to\infty}{\lim}\left\|\mathbb{D}_{1}x_{n}-\mathbb{D}_{2}y_{n}\right\| =0,\label{D3} 
\end{align}

\begin{equation}
{} \begin{array}{ll}  \underset{n\to\infty}{\lim}\left\|\mathbb{U}\mathbb{D}_{1}x_{n}-\mathbb{D}_{1}x_{n}\right\|=0, \underset{n\to\infty}{\lim}\left\|\mathbb{U}x_{n}-x_{n}\right\|=0  {\rm~and~} 

\\ \underset{n\to\infty}{\lim}\left\|\mathbb{V}\mathbb{D}_{2}y_{n}-\mathbb{D}_{2}y_{n}\right\|=0, \underset{n\to\infty}{\lim}\left\|\mathbb{V}y_{n}-y_{n}\right\|=0.& \textrm{ $  $} \label{D4} 
\end{array}.
\end{equation}
Since, $\mathbb{U}= (1-\eta)I- \eta \mathbb{T}_{1}((1-\zeta)I+\zeta \mathbb{T}_{1})I$, where $0<\eta<\zeta<\frac{1}{1+\sqrt{1+L^{2}}}$ and $\mathbb{T}_{1}$ is Lipschitz, we have
 \begin{align}
\|\eta x_{n}-\eta \mathbb{T}_{1}x_{n}\|&=\|x_{n}-(1-\eta)x_{n}- \eta \mathbb{T}_{1}x_{n}\|\nonumber
\\&=\|x_{n}-(1-\eta)x_{n}-\eta \mathbb{T}_{1}((1-\zeta)I+\zeta \mathbb{T}_{1})x_{n}+\eta \mathbb{T}_{1}((1-\zeta)I+\zeta \mathbb{T}_{1})x_{n}- \eta \mathbb{T}_{1}x_{n}\|\nonumber
\\&\leq\|x_{n}-(1-\eta)x_{n}- \eta \mathbb{T}_{1}((1-\zeta)I+\zeta \mathbb{T}_{1})x_{n}\|\nonumber
\\&+\|\eta \mathbb{T}_{1}((1-\zeta)I+\zeta \mathbb{T}_{1})x_{n}-\eta \mathbb{T}_{1}x_{n}\|\nonumber
\\&\leq \|x_{n}-\mathbb{U}x_{n}\|+\eta L\|((1-\zeta)I+\zeta \mathbb{T}_{1})x_{n}- x_{n}\|\nonumber
\\&= \|x_{n}-\mathbb{U}x_{n}\|+\eta \zeta L\|x_{n}- \mathbb{T}_{1}x_{n}\|.\nonumber
\end{align}
Therefore,
 \begin{align}
 \|x_{n}-\mathbb{T}_{1}x_{n}\|&\leq \frac{1}{(1-\zeta L)\eta}\|x_{n}-\mathbb{U}x_{n}\|.
 \end{align}
Similarly, 
 \begin{align}
 \| y_{n}-\mathbb{T}_{2}y_{n}\|&\leq \frac{1}{(1-\zeta L)\eta}\|y_{n}-\mathbb{V}y_{n}\|.
 \end{align}
By (\ref{D4}), we see  that
 \begin{align}\label{ww}
\underset{n\to\infty}{\lim}\left\|\mathbb{T}_{1}x_{n}-x_{n}\right\|=0,~ \rm {and} ~\underset{n\to\infty}{\lim}\left\|\mathbb{T}_{2}y_{n}-y_{n}\right\|=0.  
\end{align}
\subsection{Claim $(x_{n},y_{n})\rightharpoonup (x^{*},y^{*})$ }\label{section}
Since $\{\Gamma_{n}\}$ converges, it follows that $(x_{n},y_{n})$ is bounded, this implies that there exist
 $(x^{*},y^{*})\in\Gamma$ for which  $x_{n}\rightharpoonup x^{*}$ and $y_{n}\rightharpoonup y^{*}.$ 
 
 Now $v_{n}=x_{n}-\tau_{n}(\mathbb{U}x_{n}-x_{n})+\tau_{n}\mathbb{D}_{1}^{*}(\mathbb{D}_{2}y_{n}-\mathbb{D}_{1}x_{n})$ and $w_{n}=y_{n}-\tau_{n}(\mathbb{V}y_{n}-y_{n})+\tau_{n}\mathbb{D}_{2}^{*}(\mathbb{D}_{1}x_{n}-\mathbb{D}_{2}y_{n}),$ couple with (\ref{D3}) and (\ref{D4}), we deduced that
  $v_{n}\rightharpoonup x^{*}$ and $w_{n}\rightharpoonup y^{*}.$

On the other hand,  $x_{n}\rightharpoonup x^{*},$ $v_{n}\rightharpoonup x^{*},$ and $\underset{n\to\infty}{\lim}\left\|\mathbb{D}_{1}x_{n}-\mathbb{D}_{2}y_{n}\right\|=0$ together with $$v_{n}=(1-\tau_{n})x_{n}+\tau_{n}\mathbb{U} x_{n}+\tau_{n}\mathbb{D}_{1}^{*}(\mathbb{D}_{1}x_{n}-\mathbb{D}_{2}y_{n}),$$
we deduced that $x^{*}=\mathbb{U}x^{*}.$  Similarly, $y^{*}=\mathbb{V}y^{*}.$   These imply that 
 $x^{*}=\mathbb{T}_{1}x^{*}$ and $y^{*}=\mathbb{T}_{2}y^{*},$ see Lemma \ref{lemma1}. 

Now that $x_{n}\rightharpoonup x^{*}$ and  $\underset{n\to\infty}{\lim}\left\|\mathbb{T}_{1}x_{n}-x_{n}\right\|=0$ couple with demiclosedness of $(\mathbb{T}_{1}-I)$ at origin, we have  $x^{*}\in Fix(\mathbb{T}_{1}).$  

Similarly, $y_{n}\rightharpoonup y^{*}$ and $\underset{n\to\infty}{\lim}\left\|\mathbb{T}_{2}y_{n}-y_{n}\right\|=0$ together with the demiclosedness of $(\mathbb{T}_{2}-I)$ at origin, we see that  $y^{*}\in Fix(\mathbb{T}_{2}).$   

Since $v_{n}\rightharpoonup x^{*},$  $w_{n}\rightharpoonup y^{*}$ and  $\mathbb{D}_{1}$ and $\mathbb{D}_{2}$ are linear mappings, we have 
$$\mathbb{D}_{1}v_{n}\rightharpoonup \mathbb{D}_{1}x^{*}, {\rm~~and~~} \mathbb{D}_{2}w_{n}\rightharpoonup \mathbb{D}_{2}y^{*}.$$
This implies that 
$$\mathbb{D}_{1}v_{n}-\mathbb{D}_{2}w_{n}\rightharpoonup \mathbb{D}_{1}x^{*}-\mathbb{D}_{2}y^{*},$$
which turn to implies that 
$$\left\|\mathbb{D}_{1}x^{*}-\mathbb{D}_{2}y^{*}\right\|\leq\underset{n\to\infty}{\liminf}\left\|\mathbb{D}_{1}v_{n}-\mathbb{D}_{2}w_{n}\right\|=0.$$
Thus, $\mathbb{D}_{1}x^{*}=\mathbb{D}_{2}y^{*}.$ Noticing that   $(x^{*},y^{*})\in Fix(\mathbb{T}_{1})\times Fix(\mathbb{T}_{2})$, we conclude that $(x^{*},y^{*})\in \Gamma.$ 

Thus, 

\begin{enumerate}
\item[(i)]  $\underset{n\to\infty}{\lim}{\Gamma_{n}}$ exist, for all  $(x^{*}, y^{*})\in\Gamma,$ 
\item[(ii)]  $(x_{n}, y_{n})$ belong to $\Gamma.$
\end{enumerate}
Hence, by Lemma \ref{opial}, we see that  $(x_{n}, y_{n})\rightharpoonup(x^{*}, y^{*})\in\Gamma.$ Furthermore, since  $(x_{n}, y_{n})$ is bounded couple with equation (\ref{ww}) and Definition \ref{lemma01} we deduced that $(x_{n}, y_{n})\to (x^{*}, y^{*})\in\Gamma.$ This  completed the proof.
\end{proof}

\subsection{SEFPP Without   Prior   Knowledge  on  Operator Norms}

This section gives   convergent result of SEFPP without the prior knowledge  on the operator norms of  bounded and linear operators $\mathbb{D}_{1}:\mathbb{H}_{1}\to \mathbb{H}_{2}$ and $\mathbb{D}_{2}:\mathbb{H}_{2}\to \mathbb{H}_{3},$ respectively.
\begin{theorem}\label{T2} Assume  $(A1)-(A4)$ 
are satisfied and that $\Gamma\neq\emptyset.$  
Let  $\{(x_{n}, y_{n})\}$ be 
defined by 

	\begin{equation}
	{} \left\{\begin{array}{ll} \mathbb{U}=(1-\eta)I +\eta \mathbb{T}_{1}((1-\zeta)I+\zeta \mathbb{T}_{1});
		\\x_{n+1}=(1-\alpha_{n}) v_{n}+\alpha_{n}\mathbb{U}v_{n};
		\\v_{n}=(1-\tau_{n})x_{n}+\tau_{n}\mathbb{U} x_{n}+\tau_{n}\mathbb{D}_{1}^{*}(\mathbb{D}_{1}y_{n}-\mathbb{D}_{1}x_{n}), \forall n\geq 0;\\
	\\ \mathbb{V}=(1-\eta)I +\eta \mathbb{T}_{2}((1-\zeta)I+\zeta \mathbb{T}_{2});
	\\y_{n+1}=(1-\alpha_{n}) w_{n}+\alpha_{n}\mathbb{V}w_{n};
	\\w_{n}=(1-\tau_{n}) y_{n}+\tau_{n}\mathbb{V} y_{n}+\tau_{n}\mathbb{D}_{2}^{*}(\mathbb{D}_{1}x_{n}-\mathbb{D}_{2}y_{n}),\forall n\geq 0;& \textrm{ $  $} \label{al2} 
	\end{array}\right.
	\end{equation}
where $(x_{0}, y_{0})\in \mathbb{H}_{1}\times \mathbb{H}_{2}$ are chosen arbitrary and
\begin{enumerate}
	\item[(i)]  $0<a<\alpha_{n}<1;$
	\item[(ii)] $\tau_{n}\subseteq(0, 1),$ such that  $\sum{\tau_{n}}=\infty $ and   $\sum{\tau_{n}^{2}}<\infty;$
\end{enumerate} 
then, $(x_{n}, y_{n})\rightharpoonup(x^{*}, y^{*})\in\Gamma.$ In addition if $\mathbb{T}_{1}$ and $\mathbb{T}_{2}$ are   semi-compact, then  $\{(x_{n}, y_{n})\}\to (x^{*}, y^{*})\in\Gamma.$
\end{theorem} 
\begin{proof}
By   (\ref{A}), we deduced that
\begin{align}
\left\|x_{n+1}-p\right\|^{2}\leq&\left\|v_{n}-p\right\|^{2}\nonumber
\\=&\left\| x_{n}-\tau_{n}\Big(x_{n}- \mathbb{U}x_{n} - \tau_{n}\mathbb{D}_{1}^{*}(\mathbb{D}_{2}y_{n}-\mathbb{D}_{1}x_{n})\Big)-p\right\|^{2}\nonumber
\\=&\left\| x_{n}-\tau_{n}k_{n}-p\right\|^{2}, \rm{~where~} k_{n}= x_{n}-\mathbb{U}x_{n}-\mathbb{D}_{1}^{*}(\mathbb{D}_{2}y_{n}-\mathbb{D}_{1}x_{n})\nonumber
\\=&\left\|x_{n}-p\right\|^{2}-2\tau_{n}\left\langle   x_{n}-p, k_{n}\right\rangle+\tau_{n}^{2}\left\|k_{n}\right\|^{2}, \nonumber
\end{align}
thus, \begin{align}
\left\|x_{n+1}-p\right\|^{2}\leq&\left\|x_{n}-p\right\|^{2}-2\tau_{n}\left\langle   x_{n}-p, k_{n}\right\rangle+\tau_{n}^{2}\left\|k_{n}\right\|^{2}, \rm{~and~}\label{101}
\end{align}
\begin{align}
\left\langle   x_{n}-p, k_{n}\right\rangle=&\left\langle  x_{n}-p, x_{n}-\mathbb{U}x_{n}-\mathbb{D}_{1}^{*}(\mathbb{D}_{2}y_{n}-\mathbb{D}_{1}x_{n})\right\rangle\nonumber
\\=&\left\langle x_{n}-p,  x_{n}-\mathbb{U}x_{n}\right\rangle -\left\langle \mathbb{D}_{1}x_{n}-\mathbb{D}_{1}p, \mathbb{D}_{2}y_{n}-\mathbb{D}_{1}x_{n}\right\rangle\nonumber
\\\geq & \frac{1}{2}\left\|x_{n}-\mathbb{U}x_{n}\right\|^{2}-\left\langle \mathbb{D}_{1}x_{n}-\mathbb{D}_{1}p, \mathbb{D}_{2}y_{n}-\mathbb{D}_{1}x_{n}\right\rangle,\label{102} 
\end{align}
where  $\mathbb{U}$  is quasi-nonexpansive. Similarly, 
\begin{align}
\left\|y_{n+1}-q\right\|^{2}\leq\left\|y_{n}-q\right\|^{2}-2\tau_{n}\left\langle   y_{n}-q, r_{n}\right\rangle
+\tau_{n}^{2}\left\|r
_{n}\right\|^{2},\label{105} 
\end{align}
$\rm{where~} r_{n}= y_{n}-\mathbb{V}y_{n}-\mathbb{D}_{2}^{*}(\mathbb{D}_{1}x_{n}-\mathbb{D}_{2}y_{n}),$ and 
\begin{align}
\left\langle   y_{n}-q, r_{n}\right\rangle
\geq & \frac{1}{2}\left\|y_{n}-\mathbb{V}y_{n}\right\|^{2}-\left\langle \mathbb{D}_{2}y_{n}-\mathbb{D}_{2}q, \mathbb{D}_{1}x_{n}-\mathbb{D}_{2}y_{n}\right\rangle,\label{bn}
\end{align} 
  where $\mathbb{V}$ is quasi-nonexpansive. By (\ref{102}) and (\ref{bn}), and the fact that $\mathbb{D}_{1}p=\mathbb{D}_{2}q,$ we have
\begin{align}
\left\langle x_{n}-p, k_{n}\right\rangle+&\left\langle   y_{n}-q, r_{n}\right\rangle\geq\frac{1}{2}\left\|x_{n}-\mathbb{U}x_{n}\right\|^{2}-\left\langle \mathbb{D}_{1}x_{n}-\mathbb{D}_{1}p, \mathbb{D}_{2}y_{n}-\mathbb{D}_{1}x_{n}\right\rangle \nonumber
\\+&\frac{1}{2}\left\|y_{n}-\mathbb{V}y_{n}\right\|^{2}-\left\langle \mathbb{D}_{2}y_{n}-\mathbb{D}_{2}q, \mathbb{D}_{1}x_{n}-\mathbb{D}_{2}y_{n}\right\rangle\nonumber
\\=&\frac{1}{2}\Big(\left\|x_{n}-\mathbb{U}x_{n}\right\|^{2}+\left\|\mathbb{D}_{1}x_{n}-\mathbb{D}_{2}y_{n}\right\|^{2}\Big)\nonumber
\\+&\frac{1}{2}\Big(\left\|y_{n}-\mathbb{V}y_{n}\right\|^{2}+\left\|\mathbb{D}_{1}x_{n}-\mathbb{D}_{2}y_{n}\right\|^{2}\Big)\nonumber
\\\geq &\frac{1}{2}\Big(\left\|x_{n}-\mathbb{U}x_{n}\right\|^{2}+\frac{1}{\left\|\mathbb{D}_{1}\right\|^{2}}\left\|\mathbb{D}_{1}^{*}(\mathbb{D}_{1}x_{n}-\mathbb{D}_{2}y_{n})\right\|^{2}\Big)\nonumber
\\+&\frac{1}{2}\Big(\left\|y_{n}-\mathbb{V}y_{n}\right\|^{2}+\frac{1}{\left\|\mathbb{D}_{2}\right\|^{2}}\left\|\mathbb{D}_{2}^{*}(\mathbb{D}_{1}x_{n}-\mathbb{D}_{2}y_{n})\right\|^{2}\Big)\nonumber
\\\geq &\frac{1}{2\max\{1,\left\|\mathbb{D}_{1}\right\|^{2}\}}\Big(\left\|x_{n}-\mathbb{U}x_{n}\right\|^{2}+\left\|\mathbb{D}_{1}^{*}(\mathbb{D}_{1}x_{n}-\mathbb{D}_{2}y_{n})\right\|^{2}\Big)
\nonumber
\\+&\frac{1}{{2\max\{1,\left\|\mathbb{D}_{2}\right\|^{2}\}}}\Big(\left\|y_{n}-\mathbb{V}y_{n}\right\|^{2}+\left\|\mathbb{D}_{2}^{*}(\mathbb{D}_{1}x_{n}-\mathbb{D}_{2}y_{n})\right\|^{2}\Big)\nonumber
\\\geq &\frac{1}{4\max\{1,\left\|\mathbb{D}_{1}\right\|^{2},\left\|\mathbb{D}_{2}\right\|^{2}\}}\Big(\Big(\left\|y_{n}-\mathbb{V}y_{n}\right\|
+\left\|\mathbb{D}_{1}^{*}(\mathbb{D}_{1}x_{n}-\mathbb{D}_{2}y_{n})\right\|\Big)^{2}\nonumber
\\+&\Big(\left\|x_{n}-\mathbb{U}x_{n}\right\|^{2}+\left\|\mathbb{D}_{2}^{*}(\mathbb{D}_{1}x_{n}-\mathbb{D}_{2}y_{n})\right\|\Big)^{2}\Big)\nonumber
\\\geq &\eta\Big(\|k_{n}\|^{2}+\|r_{n}\|^{2}\Big),\label{kb}
\end{align}
where	$\eta=\frac{1}{4\max\{1,\left\|\mathbb{D}_{1}\right\|^{2},\left\|\mathbb{D}_{2}\right\|^{2}\}},$ $\|k_{n}\|\leq \left\|y_{n}-\mathbb{V}y_{n}\right\|+\left\|\mathbb{D}_{1}^{*}(\mathbb{D}_{1}x_{n}-\mathbb{D}_{2}y_{n})\right\|$ and $\|r_{n}\|\leq \left\|x_{n}-\mathbb{U}x_{n}\right\|+\left\|\mathbb{D}_{2}^{*}(\mathbb{D}_{1}x_{n}-\mathbb{D}_{2}y_{n})\right\|.$

By (\ref{101}), (\ref{105}) and (\ref{kb}), we have
\begin{align}
\left\|x_{n+1}-p\right\|^{2}+\left\|y_{n+1}-q\right\|^{2}\leq&\left\|x_{n}-p\right\|^{2}+\left\|y_{n}-q\right\|^{2}-2\tau_{n}\left\langle   x_{n}-p, k_{n}\right\rangle\nonumber
\\-&2\tau_{n}\left\langle   y_{n}-q, r_{n}\right\rangle+\tau_{n}^{2}\left\|k_{n}\right\|^{2} +\tau_{n}^{2}\left\|r_{n}\right\|^{2}\nonumber
\\\leq&\left\|x_{n}-p\right\|^{2}+\left\|y_{n}-q\right\|^{2}+\tau_{n}^{2}\Big(\|k_{n}\|^{2}+\|r_{n}\|^{2}\Big)\nonumber
\\-&2\tau_{n}\eta\Big(\|k_{n}\|^{2}+\|r_{n}\|^{2}\Big),\label{yz}
\\=&\left\|x_{n}-p\right\|^{2}+\left\|y_{n}-q\right\|^{2}-2\tau_{n}\big(\eta-\frac{\tau_{n}}{2}\big)\Big(\|k_{n}\|^{2}+\|r_{n}\|^{2}\Big).\nonumber
\end{align}
The fact that $\mathbb{U}$ is \rm{$L^{2}$}-Lipschitzian,  we have
\begin{align}
\left\|k_{n}\right\|=&\left\| x_{n}-\mathbb{U}x_{n}-\mathbb{D}_{1}^{*}(\mathbb{D}_{2}y_{n}-\mathbb{D}_{1}x_{n})-(p-\mathbb{U}p)+\mathbb{D}_{1}^{*}(\mathbb{D}_{2}q-\mathbb{D}_{1}p)\right\|\nonumber 
\\\leq & \left\|  x_{n}-p -\mathbb{U}( x_{n}-P) \right\| +\left\| \mathbb{D}_{1}^{*}(\mathbb{D}_{2}q-\mathbb{D}_{1}p) -\mathbb{D}_{1}^{*}(\mathbb{D}_{2}y_{n}-\mathbb{D}_{1}x_{n})\right\|\nonumber 
\\\leq &(1+L^{2})\left\|  x_{n}-p \right\|+\left\|\mathbb{D}_{1}^{*}\right\|\left\|(\mathbb{D}_{1}x_{n}-\mathbb{D}_{1}p) -(\mathbb{D}_{2}y_{n}-\mathbb{D}_{2}q)\right\|\nonumber 
\\\leq &(1+L^{2})\left\|  x_{n}-p \right\|+\left\|\mathbb{D}_{1}^{*}\right\|\Big(\left\|\mathbb{D}_{1}\right\|\left\|x_{n}-p\right\|+\left\|\mathbb{D}_{2}\right\|\left\|y_{n}-q\right\|\Big)\nonumber 
\\\leq &(1+L^{2})\left\|  x_{n}-p \right\|+\left\|\mathbb{D}_{1}^{*}\right\|\max\Big\{\left\|\mathbb{D}_{1}\right\|,\left\|\mathbb{D}_{2}\right\|\Big\}\Big(\left\|x_{n}-p\right\|+\left\|y_{n}-q\right\|\Big).\nonumber
\end{align}
This gives
\begin{align}
\left\|k_{n}\right\|\leq &(1+L^{2})\left\|  x_{n}-p \right\|+\left\|\mathbb{D}_{1}^{*}\right\|\max\Big\{\left\|\mathbb{D}_{1}\right\|,\left\|\mathbb{D}_{2}\right\|\Big\}\Big(\left\|x_{n}-p\right\|+\left\|y_{n}-q\right\|\Big).\label{as}
\end{align}
Similarly, 
\begin{align}
	\left\|r_{n}\right\|\leq &(1+L^{2})\left\|  y_{n}-q \right\|+\left\|\mathbb{D}_{2}^{*}\right\|\max\Big\{\left\|\mathbb{D}_{1}\right\|,\left\|\mathbb{D}_{2}\right\|\Big\}\Big(\left\|x_{n}-p\right\|+\left\|y_{n}-q\right\|\Big).\label{asa}
\end{align}
By (\ref{as}) and (\ref{asa}), we deduced that
\begin{align}
\left\|k_{n}\right\|+\left\|r_{n}\right\|\leq&\Big(1+L^{2}+\max(\|\mathbb{D}_{1}\|^{2},\|\mathbb{D}_{2 }\|^{2})\Big)\Big\{\left\|x_{n}-p\right\|+\left\|y_{n}-q\right\|\Big\}.
\end{align}
This gives 
\begin{align}
\max(\left\|k_{n}\right\|^{2},\left\|r_{n}\right\|^{2})\leq&\varphi^{2}\Big\{\left\|x_{n}-p\right\|+\left\|y_{n}-q\right\|\Big\}^{2}\nonumber
\\\leq&2\varphi^{2}\Big\{\left\|x_{n}-p\right\|^{2}+\left\|y_{n}-q\right\|^{2}\Big\},\label{xz}
\end{align}
{\rm~where~} $\varphi:=1+L^{2}+\max(\|\mathbb{D}_{1}\|^{2},\|\mathbb{D}_{2}\|^{2}).$

Equation (\ref{yz}) and  (\ref{xz}) gave
\begin{align}
\left\|x_{n+1}-p\right\|^{2}+\left\|y_{n+1}-q\right\|^{2} 
\leq&\left\|x_{n}-p\right\|^{2}+\left\|y_{n}-q\right\|^{2}\nonumber
\\+&2\varphi^{2}\tau_{n}^{2}\Big\{\left\|x_{n}-p\right\|^{2}+\left\|y_{n}-q\right\|^{2}\Big\}.\label{yz1}
\end{align}

\noindent Noticing that $\underset{n}{\sum}{\varphi^{2}\tau_{n}^{2}<\infty},$ by Lemma \ref{lemmat} we deduced that $\underset{n\to\infty}{\lim}\Big\{\left\|x_{n}-p\right\|^{2}+\left\|y_{n}-q\right\|^{2}\Big\}$ exist. Thus, $(x_{n},y_{n})$ is bounded, and it is not difficult  to see that $\Big(\|k_{n}\|^{2}+\|r_{n}\|^{2}\Big)$ is bounded.\\

\noindent Next, we show that  $\underset{n\to\infty}{\lim}\left\|\mathbb{U}x_{n}-x_{n}\right\|=0,$  $\underset{n\to\infty}{\lim}\left\|\mathbb{V}y_{n}-y_{n}\right\|=0$  and  $\underset{n\to\infty}{\lim}\left\|\mathbb{D}_{1}x_{n}-\mathbb{D}_{2}y_{n}\right\|=0.$ \\  
 
\noindent By  (\ref{yz}), we have 
\begin{align}
2\tau_{n}\eta\Big(\|k_{n}\|^{2}+\|r_{n}\|^{2}\Big)
&\leq\Big(\left\|x_{n}-p\right\|^{2}-\left\|x_{n+1}-p\right\|^{2}\Big)+\Big(\left\|y_{n}-q\right\|^{2}-\left\|y_{n+1}-q\right\|^{2}\Big)\nonumber
\\&+\tau_{n}^{2}\Big(\|k_{n}\|^{2}+\|r_{n}\|^{2}\Big),\label{nn1}
\end{align} 
thus,
 \begin{align}
 2\eta\underset{n}{\sum}\tau_{n}\Big(\|k_{n}\|^{2}+\|r_{n}\|^{2}\Big)
 &\leq\underset{n}{\sum}\Big(\left\|x_{n}-p\right\|^{2}-\left\|x_{n+1}-p\right\|^{2}+\left\|y_{n}-q\right\|^{2}-\left\|y_{n+1}-q\right\|^{2}\Big)\nonumber
 \\&+\underset{n=0}{\sum}\tau_{n}^{2}\Big(\|k_{n}\|^{2}+\|r_{n}\|^{2}\Big)\nonumber
 \\&=\left\|x_{0}-p\right\|^{2}-\left\|x_{1}-p\right\|^{2}+\left\|y_{0}-q\right\|^{2}-\left\|y_{1}-q\right\|^{2}\nonumber
\\ &+\left\|x_{1}-p\right\|^{2}-\left\|x_{2}-p\right\|^{2}+\left\|y_{1}-q\right\|^{2}-\left\|y_{2}-q\right\|^{2}\nonumber
\\&+\left\|x_{2}-p\right\|^{2}-\left\|x_{3}-p\right\|^{2}+\left\|y_{2}-q\right\|^{2}-\left\|y_{3}-q\right\|^{2}\nonumber
 \\& .\nonumber
 \\& .\nonumber
 \\& .\nonumber
 \\&+\underset{n=0}{\sum}\tau_{n}^{2}\Big(\|k_{n}\|^{2}+\|r_{n}\|^{2}\Big)\nonumber
 \\&=\left\|x_{0}-p\right\|^{2}+\left\|y_{0}-q\right\|^{2}+\underset{n=0}{\sum}\tau_{n}^{2}\Big(\|k_{n}\|^{2}+\|r_{n}\|^{2}\Big).\label{nn3}
 \end{align}
 Since $\underset{n}{\sum}\tau_{n}=\infty,$ $\underset{n}{\sum}\tau_{n}^{2}<\infty$ and $\Big(\|k_{n}\|^{2}+\|r_{n}\|^{2}\Big)$ is bounded, thus, we deduced from   (\ref{nn3}) that 
 \begin{equation}\underset{n\to \infty}{\limsup}\Big(\|k_{n}\|^{2}+\|r_{n}\|^{2}\Big)=0.\label{nn9}
 \end{equation}
 
 On the other hand, 
 \begin{align}
 \left\|k_{n+1}-k_{n}\right\|
 &\leq  \left\|  x_{n+1}-x_{n} -\mathbb{U}( x_{n+1}-x_{n}) \right\|\nonumber
 \\& +\left\| \mathbb{D}_{1}^{*}(\mathbb{D}_{2}y_{n}-\mathbb{D}_{1}x_{n}) -\mathbb{D}_{1}^{*}(\mathbb{D}_{2}y_{n+1}-\mathbb{D}_{1}x_{n+1})\right\|\nonumber 
 \\&\leq(1+L^{2})\left\|   x_{n+1}-x_{n} \right\|+\left\|\mathbb{D}_{1}^{*}\right\|\left\|( \mathbb{D}_{1}x_{n+1}-\mathbb{D}_{1}x_{n}) -(\mathbb{D}_{2}y_{n+1}-\mathbb{D}_{2}y_{n})\right\|\nonumber 
 \\&\leq(1+L^{2})\left\|   x_{n+1}-x_{n} \right\|+\left\|\mathbb{D}_{1}^{*}\right\|\Big(\left\|\mathbb{D}_{1}\right\|\left\|x_{n+1}-x_{n}\right\|+\left\|\mathbb{D}_{2}\right\|\left\|y_{n+1}-y_{n}\right\|\Big)\nonumber 
 \\&\leq(1+L^{2})\left\|  x_{n+1}-x_{n}\right\|\nonumber
 \\&+\left\|\mathbb{D}_{1}^{*}\right\|\max\Big\{\left\|\mathbb{D}_{1}\right\|,\left\|\mathbb{D}_{2}\right\|\Big\}\Big(\left\| x_{n+1}-x_{n}\right\|+
 \left\|y_{n+1}-y_{n}\right\|\Big),\label{nn4}
 \end{align}
and
 \begin{align}\left\|  x_{n+1}-x_{n} \right\|\leq \left\| \tau_{n} k_{n} \right\|\rm{~~and~~}   \left\|  y_{n+1}-y_{n} \right\|\leq \left\| \tau_{n} r_{n} \right\|.\label{nn5}
  \end{align} 
 By  (\ref{nn4}) and (\ref{nn5}), we deduced that 
  \begin{align}
  \left\|k_{n+1}-k_{n}\right\|&\leq(1+L^{2})\tau_{n} \left\|  k_{n} \right\|\nonumber
  \\&+\left\|\mathbb{D}_{1}^{*}\right\|\max\Big\{\left\|\mathbb{D}_{1}\right\|,\left\|\mathbb{D}_{2}\right\|\Big\}\tau_{n}\Big( \left\|k_{n} \right\|+
   \left\|  r_{n} \right\|\Big).\label{nn66}
  \end{align}
  Similarly,   
  \begin{align}
  \left\|r_{n+1}-r_{n}\right\|&\leq(1+L^{2})
  \tau_{n} \left\|  r_{n} \right\|\nonumber
  \\&+\left\|\mathbb{D}_{2}^{*}\right\|\max\Big\{\left\|\mathbb{D}_{1}\right\|,\left\|\mathbb{D}_{2}\right\|\Big\}
  \tau_{n}\Big( \left\|r_{n}\right\|+
  \left\| k_{n} \right\|\Big).\label{nn67}
  \end{align}
  By (\ref{nn66}) and (\ref{nn67}), we have
  \begin{align}
  \max\Big(\left\|k_{n+1}-k_{n}\right\|,\left\|r_{n+1}-r_{n}\right\|\Big)&\leq(1+L^{2})\tau_{n} \Big(\left\|k_{n} \right\|+\left\| r_{n} \right\|\Big)\nonumber
  \\&+\left\|\mathbb{D}_{1}^{*}\right\|\max\Big\{\left\|\mathbb{D}_{1}\right\|^{2},\left\|\mathbb{D}_{2}\right\|^{2}\Big\}\tau_{n}\Big(\left\|k_{n} \right\|+\left\| r_{n} \right\|\Big)\nonumber
  \\&=\tau_{n} j_{\mathbb{D}_{1},\mathbb{D}_{2}}\Big(\left\|k_{n} \right\|+\left\| r_{n} \right\|\Big), \nonumber
   \end{align}
   where $j_{\mathbb{D}_{1},\mathbb{D}_{2}} =\Big((1+L^{2}) +\left\|\mathbb{D}_{1}^{*}\right\|
   \max\Big\{\left\|\mathbb{D}_{1}\right\|^{2},\left\|\mathbb{D}_{2}\right\|^{2}\Big\}\Big).$
   Thus \begin{align}
   \max\Big(\left\|k_{n+1}-k_{n}\right\| ,\left\|r_{n+1}-r_{n}\right\| \Big)&\leq
   \tau_{n}  j_{\mathbb{D}_{1},\mathbb{D}_{2}} \Big(\left\|k_{n} \right\|+\left\| r_{n} \right\|\Big) . \label{nn7}  
   \end{align}
   
   By   (\ref{nn7}), we deduced that
     \begin{align}
   \left\langle  k_{n}, k_{n+1}- k_{n}\right\rangle+\left\langle  r_{n}, r_{n+1}- r_{n}\right\rangle&\leq \|k_{n}\|\| k_{n+1}- k_{n}\| +\| r_{n}\|\| r_{n+1}- r_{n}\|\nonumber
   \\&\leq \tau_{n}^{2} j_{\mathbb{D}_{1},\mathbb{D}_{2}}^{2}\Big(\left\|k_{n} \right\|+\left\| r_{n} \right\|\Big)^{2}\nonumber
   \\&\leq 2\tau_{n}^{2} j_{\mathbb{D}_{1},\mathbb{D}_{2}}^{2}\Big(\left\|k_{n} \right\|^{2}+\left\| r_{n} \right\|^{2}\Big).\label{nn8}
   \end{align}
    On the other hand, 
   \begin{align}
   \|k_{n+1}\|^{2}+\|r_{n+1}\|^{2}&=\|k_{n}\|^{2}+ \left\langle  k_{n}, k_{n+1}- k_{n}\right\rangle +\|k_{n+1}-k_{n}\|^{2}\nonumber
   \\&+\|r_{n}\|^{2}+\left\langle  r_{n}, r_{n+1}- r_{n}\right\rangle+ \|r_{n+1}- r_{n}\|^{2}\nonumber
   \\&\leq \Big(1+2\tau_{n}^{2} j_{\mathbb{D}_{1},\mathbb{D}_{2}}^{2}+4\tau_{n}^{2} j_{\mathbb{D}_{1},\mathbb{D}_{2}}^{2} \Big)\Big(\left\|k_{n} \right\|^{2}+\left\| r_{n} \right\|^{2}\Big).
   \end{align}
   Noticing that $\underset{n}{\sum}(6\tau_{n}^{2} j_{\mathbb{D}_{1},\mathbb{D}_{2}}^{2})<\infty,$ thus, by Lemma \ref{lemmat}, we deduced that 
    \begin{align}\underset{n\to\infty}{\lim}\Big(\left\|k_{n} \right\|^{2}+\left\| r_{n} \right\|^{2}\Big) ~\rm{exist,~  this~ imply~ that }
      \end{align}
      \begin{align}\underset{n\to\infty}{\lim}(\left\|k_{n} \right\|+\left\| r_{n} \right\|\Big) ~\rm{exist}.\label{nn10}
      \end{align}
     Therefore, we deduced from (\ref{nn9}) and  (\ref{nn10}) that 
      \begin{align}\underset{n\to\infty}{\lim}\left\|k_{n} \right\|&=\underset{n\to\infty}{\lim}\left\|   x_{n}-\mathbb{U}x_{n}-\mathbb{D}_{1}^{*}(\mathbb{D}_{2}y_{n}-\mathbb{D}_{1}x_{n})\right\|=0\nonumber ~  \rm{and}
     \\ \underset{n\to\infty}{\lim}\left\|r_{n} \right\|&=\underset{n\to\infty}{\lim}\left\|   y_{n}-\mathbb{V}y_{n}-\mathbb{D}_{2}^{*}(\mathbb{D}_{1}x_{n}-\mathbb{D}_{2}y_{n})\right\|=0.\label{nn12}
      \end{align}
     It is known that  $\mathbb{U}$ is quasi nonexpansive mapping if and only if
     $\frac{1}{2}\|x_{n}-\mathbb{U}x_{n}\|^{2}\leq \left\langle x_{n}-\mathbb{U}x_{n}, x_{n}-p \right\rangle, p\in Fix(\mathbb{U}).$ This implies that
     \begin{align}
     \frac{1}{2}\|x_{n}-\mathbb{U}x_{n}\|^{2}&+\left\langle \mathbb{D}_{1}x_{n}-\mathbb{D}_{2}y_{n}, \mathbb{D}_{1}x_{n}-\mathbb{D}_{1}p \right\rangle\leq \left\langle x_{n}-\mathbb{U}x_{n}, x_{n}-p \right\rangle\nonumber
     \\&+ \left\langle \mathbb{D}_{1}x_{n}-\mathbb{D}_{2}y_{n}, \mathbb{D}_{1}x_{n}-\mathbb{D}_{1}p \right\rangle\nonumber
     \\&= \left\langle x_{n}-\mathbb{U}x_{n}-\mathbb{D}_{1}^{*}(\mathbb{D}_{1}x_{n}-\mathbb{D}_{2}y_{n}), x_{n}-p \right\rangle\nonumber
     \\&\leq \|x_{n}-\mathbb{U}x_{n}-\mathbb{D}_{1}^{*}(\mathbb{D}_{1}x_{n}-\mathbb{D}_{2}y_{n})\|\|x_{n}-p\|\label{nn13}.
     \end{align}
  Thus, by  (\ref{nn12}) we deduced that 
     \begin{align}
       \underset{n\to\infty}{\lim}\left\|x_{n}-\mathbb{U}x_{n} \right\|=0.\label{nn14}
     \end{align}
   Similarly, 
   \begin{align}
   \underset{n\to\infty}{\lim}\left\|y_{n}-\mathbb{V}y_{n} \right\|=0,\label{nn151}
   \end{align}
and
\begin{align}
\left\|\mathbb{D}_{1}x_{n}-\mathbb{D}_{2}y_{n} \right\|^{2}&=\left\langle \mathbb{D}_{1}x_{n}-\mathbb{D}_{2}y_{n}, \mathbb{D}_{1}x_{n}-\mathbb{D}_{2}y_{n} \right\rangle\nonumber
\\&=\left\langle \mathbb{D}_{1}x_{n}-\mathbb{D}_{2}y_{n}, \mathbb{D}_{1}x_{n}-\mathbb{D}_{1}p \right\rangle
+\left\langle \mathbb{D}_{2}y_{n}-\mathbb{D}_{1}x_{n},\nonumber \mathbb{D}_{2}y_{n}-\mathbb{D}_{2}q \right\rangle,
\\&=\left\langle \mathbb{D}_{1}^{*}(\mathbb{D}_{1}x_{n}-\mathbb{D}_{2}y_{n}), x_{n}-p \right\rangle+\left\langle \mathbb{D}_{2}^{*}(\mathbb{D}_{2}y_{n}-\mathbb{D}_{1}x_{n}), y_{n}-q \right\rangle\nonumber
\\&\leq \|\mathbb{D}_{1}^{*}(\mathbb{D}_{1}x_{n}-\mathbb{D}_{2}y_{n})+(x_{n}-Ux_{n})-(x_{n}-\mathbb{U}x_{n})\|\|x_{n}-p\|\|\nonumber
\\&+\|\mathbb{D}_{2}^{*}(\mathbb{D}_{2}y_{n}-\mathbb{D}_{1}x_{n})+(y_{n}-mathbb{V}y_{n})-(y_{n}-\mathbb{V}y_{n})\|\|y_{n}-q\|\nonumber
\\&\leq\|\mathbb{D}_{1}^{*}(\mathbb{D}_{1}x_{n}-\mathbb{D}_{2}y_{n})-(x_{n}-\mathbb{U}x_{n})\|\|x_{n}-p\|
+\|(x_{n}-\mathbb{U}x_{n})\|\|x_{n}-p\|\nonumber
\\&+\|\mathbb{D}_{2}^{*}(\mathbb{D}_{2}y_{n}-\mathbb{D}_{1}x_{n})-(y_{n}-\mathbb{V}y_{n})\|\|y_{n}-q\|
+\|(y_{n}-Vy_{n})\|\|y_{n}-q\|
   \end{align}
  Thus,  by (\ref{nn14}), (\ref{nn151}) and the fact that $\{\|x_{n}-p\|\}$ and $\{\|y_{n}-q\|\}$ are bounded, we deduce that
   \begin{align}
   \underset{n\to\infty}{\lim}\left\|\mathbb{D}_{1}x_{n}-\mathbb{D}_{2}y_{n} \right\|=0.\label{nn15}
   \end{align} 
   Next, we show that $(x_{n},y_{n})\rightharpoonup (x^{*}, y^{*})\in\Gamma.$ This follows directly from Subsection \ref{section}.  This completed the proof.
\end{proof}

   \begin{cor}
    Suppose for  $j=1,2, \mathbb{T}_{j}:H_{j}\to H_{j}$  are $k_{j}-$demicontractive mappings with $Fix(\mathbb{T}_{j})\neq\emptyset$  such that $(\mathbb{T}_{j}-I)$  are demiclosed at zero.
   		Let $(x_{n},y_{n})\in \mathbb{H}_{1}\times \mathbb{H}_{2},$ be generated by 
   		\begin{equation}
   		{} \left\{\begin{array}{ll} x_{n+1}=(1-\alpha_{n}) v_{n}+\alpha_{n}\mathbb{T}_{1}v_{n};
   			\\v_{n}=(1-\tau_{n})x_{n}+\tau_{n}\mathbb{T}_{1} x_{n}+\tau_{n}\mathbb{D}_{1}^{*}(\mathbb{D}_{2}y_{n}-\mathbb{D}_{1}x_{n}), \forall n\geq 0;\\
   		\\ y_{n+1}=(1-\alpha_{n}) w_{n}+\alpha_{n}\mathbb{T}_{2}w_{n};
   		\\w_{n}=(1-\tau_{n}) y_{n}+\tau_{n}\mathbb{T}_{2} y_{n}+\tau_{n}\mathbb{D}_{2}^{*}(\mathbb{D}_{1}x_{n}-\mathbb{D}_{1}y_{n}), \forall n\geq 0;& \textrm{ $  $} \label{500} 
   		\end{array}\right.
   		\end{equation}
   		 where $(x_{0},y_{0})\in \mathbb{H}_{1}\times \mathbb{H}_{2},$  are  chosen arbitrary, $0<a<\alpha_{n}<1,$ and $\tau_{n}\in\left(0, \frac{2}{L_{1}+L_{2}}\right),$ where   $L_{1}=\mathbb{D}_{1}^{*}\mathbb{D}_{1}$ and $L_{2}=\mathbb{D}_{2}^{*}\mathbb{D}_{2},$ respectively. 	Then, $\{(x_{n}, y_{n})\}$ defined  by  (\ref{500}) converges   to $\Gamma.$
   	\end{cor}

  \begin{cor} (Chang et al.,\cite{2})  Suppose $(A1)-(A4)$ are satisfied,  
    and that  $\Gamma\neq\emptyset.$    Let   $\{(x_{n}, y_{n})\}$ 
  	defined by  \begin{equation}
  		{} \left\{\begin{array}{ll}
  			x_{n+1}=\alpha_{n}x_{n}+(1-\alpha_{n})\Big((1-\eta_{n})I +\eta_{n} \mathbb{T}_{1}((1-\zeta_{n})I+\zeta_{n} \mathbb{T}_{1})\Big)v_{n} ;
  			\\v_{n}=x_{n}+\tau_{n}\mathbb{D}_{1}^{*}(\mathbb{D}_{2}y_{n}-\mathbb{D}_{1}x_{n}), \forall n\geq 0.\\
  			
  			\\ y_{n+1}=\alpha_{n}y_{n}+(1-\alpha_{n})\Big((1-\eta_{n})I +\eta_{n} \mathbb{T}_{2}((1-\zeta_{n})I+\zeta_{n} \mathbb{T}_{2})\Big)w_{n} ;
  			\\w_{n}= y_{n}+\tau_{n}\mathbb{D}_{2}^{*}(\mathbb{D}_{1}x_{n}-\mathbb{D}_{2}y_{n}), \forall n\geq 0.& \textrm{ $  $} \label{0011} 
  		\end{array}\right.
  	\end{equation}
  	 where $0<\eta_{n}<\zeta_{n}<\frac{1}{1+\sqrt{1+L^{2}}},$  $0<a<\alpha_{n}<1,$ and $\tau_{n}\in\left(0, \frac{2}{L_{1}+L_{2}}\right)$ with  $L_{1}=\mathbb{D}_{1}^{*}\mathbb{D}_{1},$ $L_{2}=\mathbb{D}_{2}^{*}\mathbb{D}_{2},$ and $(x_{0}, y_{0}) \in \mathbb{H}_{1}\times H_{2}$ were chosen arbitrary.
  Then 	$\{(x_{n}, y_{n})\}$ converges  to  $\Gamma.$
  \begin{proof} 
  	Algorithm (\ref{0011}) is a special case of algorithm (\ref{al1}) by taking $\tau_{n}=1$, $Vy_{n}=y_{n}$ and $\beta_{n}=(1-\alpha_{n}).$ Therefore, the proof of this corollary  follows directly from Theorem \ref{T1}.
  \end{proof} 
  \end{cor}
  \begin{cor} (Chang et al.,\cite{a})  Suppose $(A1)-(A4)$ are satisfied,  
  	and that  $\Gamma\neq\emptyset.$    Let   $\{(x_{n}, y_{n})\}$ 
  	defined by  \begin{equation}
  		{} \left\{\begin{array}{ll}
  		U	=(1-\eta_{n})I +\eta_{n} \mathbb{T}_{1}((1-\zeta_{n})I+\zeta_{n} \mathbb{T}_{1})  ;
  			\\x_{n+1}=x_{n}-\tau_{n}\Big( x_{n}-Ux_{n} +\mathbb{D}_{1}^{*}(\mathbb{D}_{2}y_{n}-\mathbb{D}_{1}x_{n})\Big), \forall n\geq 0.\\
  			
  			\\V=(1-\eta_{n})I +\eta_{n} \mathbb{T}_{2}((1-\zeta_{n})I+\zeta_{n} \mathbb{T}_{2});
  			\\ y_{n+1}= y_{n}-\tau_{n}\Big( y_{n}-Vy_{n}+\mathbb{D}_{2}^{*}(\mathbb{D}_{1}x_{n}-\mathbb{D}_{2}y_{n})\Big), \forall n\geq 0.& \textrm{ $  $} \label{011} 
  		\end{array}\right.
  	\end{equation}
  	where $0<\eta_{n}<\zeta_{n}<\frac{1}{1+\sqrt{1+L^{2}}},$  $0<a<\alpha_{n}<1,$ and $\tau_{n}\in\left(0, \frac{2}{L_{1}+L_{2}}\right)$ with  $L_{1}=\mathbb{D}_{1}^{*}\mathbb{D}_{1},$ $L_{2}=\mathbb{D}_{2}^{*}\mathbb{D}_{2},$ and $x_{0}\in \mathbb{H}_{1}$ was chosen arbitrary.
  	Then 	$\{(x_{n}, y_{n})\}$ converges  to  $\Gamma.$
  	 \begin{proof}
  		Algorithm (\ref{011}) is a special case of algorithm (\ref{al1}) by taking $\alpha_{n}=0.$ Therefore, the proof of this corollary follows directly from Theorem \ref{T1}.
  	\end{proof} 
\end{cor}
   \section{APPLICATIONS}
   This section provides applications of  SEFPP.
   \subsection{Application to SFP}
 We denote the solution set of SFP   (\ref{eqn1}) by  
   $\Gamma_{1} $
   In  (\ref{eqn21}), if  $ \mathbb{D}_{2}=I$ (identity mapping),  then SEFPP reduces to   (\ref{eqn1}). Furthermore, in  algorithm (\ref{al1}), let $y_{n}:=\mathbb{T}_{1}\mathbb{D}_{1}x_{n},$ we therefore,  deduced the following theorem:
   
   \begin{theorem}\label{T4} Let $\mathbb{H}_{j},j=1,2$   $\mathbb{K}_{j}, j=1,2,$   $\mathbb{D}_{1},$ $\mathbb{D}_{1}^{*},$ $\mathbb{T}$ and $(\mathbb{T}-I) $  be as in Theorem \ref{T1}. Let $\{x_{n}\}$
    be defined by 
   	\begin{equation}
   	{} \left\{\begin{array}{ll} x_{n+1}=(1-\alpha_{n}) v_{n}+\alpha_{n}\mathbb{U}v_{n};
   	\\v_{n}=(1-\tau_{n})x_{n}+\tau_{n}\mathbb{U} x_{n}+\tau_{n}\mathbb{D}^{*}(\mathbb{T}\mathbb{D}x_{n}-\mathbb{D}x_{n});
   	\\ \mathbb{U}=(1-\eta)I +\eta \mathbb{T}((1-\zeta)I+\zeta \mathbb{T});
   	\end{array}\right.
   	\end{equation}\label{SFP1}
   where $0<\eta<\zeta<\frac{1}{1+\sqrt{1+L^{2}}},$  $0<a<\alpha_{n}<1,$ and $\tau_{n}\in\left(0, \frac{2}{L_{1}}\right)$ with  $L_{1}=\mathbb{D}_{1}^{*}\mathbb{D}_{1},$ and $x_{0}\in \mathbb{H}_{1}$ was chosen arbitrary.
    Then, $\{x_{n}\}$ converges  weakly to $\Gamma_{1}.$
   \end{theorem}

   \subsection{Application to Split Variational	Inequality Problem (SVIP)}
   The SVIP was introduced by Censor et al., \cite{c} and it entails as finding
  \begin{align}
  x^{*} \in \mathbb{K}_{1} \rm{~such ~that~} & \left\langle \mathbb{T}_{1}(x^{*}),  x-x^{*} \right\rangle \geq 0, \rm{\forall} x\in \mathbb{K}_{1}, \label{vip} 
\end{align}  
 \begin{align} \rm{and~} y^{*} =  \mathbb{D}_{1} x^{*} \in \mathbb{K}_{2} \rm{~ solves~}& \left\langle \mathbb{T}_{2}(y^{*}), y-y^{*} \right\rangle \geq 0,  \rm{\forall~} y\in \mathbb{K}_{2}, 
  \end{align}
 where $ \mathbb{T}_{j} : \mathbb{H}_{j} \to \mathbb{H}_{j},j=1,2.$
are some nonlinear mappings.
  
  Equation (\ref{vip})   is called variational inequality problem (VIP), and we denoted its solution set   by $VI(\mathbb{K}_{1},\mathbb{T}_{1}).$ Subsequently, the solution set of SVIP would be denoted by
  
\begin{equation}  
   \Gamma_{2}=\Big\{\rm (x^{*},y^{*})\times VI(\mathbb{K}_{1},\mathbb{T}_{1})\times VI(\mathbb{K}_{2},\mathbb{T}_{2})\in \mathbb{K}_{1} \rm{~such ~that ~} \mathbb{D}_{1}x^{*}=\mathbb{D}_{2}y^{*}\Big\}.\label{SVI}
\end{equation}
   Let $g_{1}(y,t):\mathbb{K}_{1}\times \mathbb{K}_{2}\to \mathbb{R}$ be defined by 
   $$g_{1}(y,t)=\left\langle \mathbb{T}_{1}y, t-y\right\rangle, \forall t,y\in \mathbb{K}_{1}, $$  
   and 
   $g_{2}(x,t):\mathbb{K}_{2}\times \mathbb{K}_{2}\to \mathbb{R}$  be defined by $$ g_{2}(x,t)=\left\langle \mathbb{T}_{2}x, t-x\right\rangle, \forall x,t\in \mathbb{K}_{2}.$$
   
   For $\lambda>0,$ the re-solvent operators of $g_{1}$ and $g_{2}$ are denoted by $R_{\lambda,g_{1}}$ and $R_{\lambda,g_{2}}$ and  are defined by
    \begin{equation}\label{5.5}
    R_{\lambda,g_{1}}(q)=\Big\{y\in \mathbb{K}_{1}: g_{1}(y,t)+\frac{1}{\lambda}\left\langle  t-y,y-q  \right\rangle, \forall t\in \mathbb{K}_{1} \Big\},
\end{equation}  and 
 \begin{equation}\label{5.6}R_{\lambda,g_{2}}(p)=\Big\{x\in \mathbb{K}_{2}: g_{2}(x,t)+\frac{1}{\lambda}\left\langle  t-x,x-p\right\rangle, \forall t\in \mathbb{K}_{2} \Big\},
\end{equation} respectively. 
    
    It was proved in \cite{a} that $R_{\lambda,g_{1}}$ and $R_{\lambda,g_{2}}$ are quasi-pseudocontractive  and 1-Lipschitzian mappings with 
    $Fix(R_{\lambda,g_{1}}) = VI(\mathbb{K}_{1},g_{1}) \neq\emptyset$ and  $Fix(R_{\lambda,g_{2}}) = VI(\mathbb{K}_{2},g_{2}) \neq\emptyset.$ Therefore, the SVIP  is equivalent to
    the following split equality fixed point problem:
    \begin{equation}
   \rm{Find~}  x^{*} \in Fix(R_{\lambda,g_{1} })\rm{~and~} y^{*}\in Fix(R_{\lambda,g_{2}})  \rm{~such ~that~} \mathbb{D}_{1}x^{*} = \mathbb{D}_{2}y^{*}.
    \end{equation} Hence, we have the following result: 

Suppose that
  \begin{enumerate}
  	\item [(B1)] $R_{\lambda,g_{1}}$ and  $R_{\lambda , g_{2}}$ be defined as in (\ref{5.5}) and (\ref{5.6}); 
  	\item [(B2)]  $\mathbb{D}_{j},j=1,2$ and $\mathbb{D}_{j}^{*},j=1,2$ as in Theorem \ref{1}.
  	\item [(B3)] $(R_{\lambda,g_{1}}-I)$ and $(R_{\lambda,g_{2}}-I)$ are demiclosed at zero.
  	\item [(B4)] Let $U$ and $V$ be defined as follows: 
  	\begin{equation}\nonumber
  	{} \left\{\begin{array}{ll} \mathbb{U}=(1-\eta)I +\eta R_{\lambda,g_{1}}((1-\zeta)I+\zeta R_{\lambda,g_{1}}),
  	\\\mathbb{V}=(1-\eta)I +\eta R_{\lambda,g_{2}}((1-\zeta)I+\zeta R_{\lambda,g_{2}}), & \textrm{ $  $} \label{3} 
  	\end{array}\right.
  	\end{equation}
  	where $0<\eta<\zeta<\frac{1}{1+\sqrt{1+L^{2}}}.$
  	\item [(B5)] \textbf{Algorithm:} Let $(x_{n},y_{n})\in \mathbb{H}_{1}\times \mathbb{H}_{2}$ be defined by 
   
  	\begin{equation}
  	{} \left\{\begin{array}{ll} x_{n+1}=(1-\alpha_{n}) v_{n}+\alpha_{n}\mathbb{U}v_{n};
  		\\v_{n}=(1-\tau_{n})x_{n}+\tau_{n}U x_{n}+\tau_{n}\mathbb{D}_{1}^{*}(\mathbb{D}_{2}y_{n}-\mathbb{D}_{1}x_{n}), n\geq 0.\\
  		
  	\\ y_{n+1}=(1-\alpha_{n}) w_{n}+\alpha_{n}\mathbb{V}w_{n};
  	\\w_{n}=(1-\tau_{n}) y_{n}+\tau_{n}\mathbb{V} y_{n}+\tau_{n}\mathbb{D}_{2}^{*}(\mathbb{D}_{1}x_{n}-\mathbb{D}_{2}y_{n}), \forall n\geq 0.& \textrm{ $  $} \label{5j} 
  	\end{array}\right.
  	\end{equation}
  	
  	where $(x_{0},y_{0})\in \mathbb{H}_{1}\times \mathbb{H}_{2}$ are  chosen arbitrary, $0<a<\alpha_{n}<1,$ and $\tau_{n}\in\left(0, \frac{2}{L_{1}+L_{2}}\right),$ with   $L_{1}=\mathbb{D}_{1}^{*}\mathbb{D}_{1}$ and $L_{2}=\mathbb{D}_{2}^{*}D_{2},$ respectively.
  \end{enumerate}
  
  \begin{theorem}\label{Tsvip} Suppose that  assumptions $(B1)-(B6)$ 
  	are satisfied and that $\Gamma_{2}\neq\emptyset.$  
  	Then, the sequence $\{(x_{n}, y_{n})\}$ 
  	generated by  (\ref{5j}) 
  	converges  to the solution of SVIP.
  \end{theorem}
    
  \subsection{Application to the  Split Convex Minimization Problem(SCMP)}  
  Let $\mathbb{M} : \mathbb{K}_{1}\to \mathbb{R}$ and $\mathbb{N} : \mathbb{K}_{2}\to \mathbb{R}$ be  lower semi-continuous and proper convex functions.  The  SCMP is formulated as follows: 
  $$\mathbb{M}(x^{*})=\underset{x\in \mathbb{K}_{1}}{\min}\mathbb{M}(x),  \mathbb{N}(y^{*})=\underset{y\in \mathbb{K}_{2}}{\min}\mathbb{N}(y) \rm{~such ~that~} \mathbb{D}_{1}x^{*}=\mathbb{D}_{2}y^{*}.$$
  We denoted the solution SET of  SCMP by 
  \begin{align}
  \Gamma_{3}=&\Big\{\rm  (x^{*},y^{*})\in \mathbb{K}_{1}\times \mathbb{K}_{2},   \mathbb{M}(x^{*})=\underset{x\in \mathbb{K}_{1}}{\min}\mathbb{M}(x),  \mathbb{N}(y^{*})=\underset{y\in \mathbb{K}_{2}}{\min}\mathbb{N}(y)\nonumber
  \\&\rm{such ~that~} \mathbb{D}_{1}x^{*}=\mathbb{D}_{2}y^{*}\Big\}.\label{DSVI}
  \end{align}
  Let $h(x,y): \mathbb{K}_{1}\times \mathbb{K}_{1} \to\mathbb{R}$ and $l(r,t):\mathbb{K}_{2}\times \mathbb{K}_{2} \to\mathbb{R}$  be defined by
  $h(x,y)=\mathbb{M}(x)-\mathbb{N}(y)$ and $l(r,t)=\mathbb{M}(r)-\mathbb{N}(t),$ respectively. For arbitrary $\lambda >0,$ Chang et al \cite{a}, defined the re-solvent operators of $\lambda,$ h and k as follows:
  $$R_{\lambda,h}(x)=\Big\{t\in C: f(t,y)+\frac{1}{\lambda}\left\langle  y-t,t-x\right\rangle \Big\} $$  and $$R_{\lambda,l}(y)=\Big\{t\in C: f(t,x)+\frac{1}{\lambda}\left\langle  x-t,t-y\right\rangle \Big\}.$$ It was proved in \cite{a} that
  $Fix(R_{\lambda,h})=\Big\{\mathbb{M}(x^{*})=\underset{x\in \mathbb{K}_{1}}{\min}\mathbb{M}(x)\Big\}$ and 
  $Fix(R_{\lambda,l})=\Big\{ \mathbb{N}(y^{*})=\underset{y\in \mathbb{K}_{2}}{\min}\mathbb{N}(y) \Big\}.$ 
  Therefore, SCMP  for $\mathbb{M}$ and $\mathbb{N}$ is equivalent to the following SEFPP:
 \begin{equation}\Gamma_{3}=\{ \rm{find~}  x^{*}\in Fix(R_{\lambda,h}) \rm{ ~and~} y^{*} \in Fix(R_{\lambda,l})\rm{~such ~that~} \mathbb{D}_{1}x^{*} = \mathbb{D}_{2}y^{*}\}.
 \end{equation}
  It was proved in \cite{a} that  $R_{\lambda,h}$ and $R_{\lambda,l}$ are firmly nonexpansive with $Fix(R_{\lambda,h})\neq\emptyset $ and $Fix(R_{\lambda,k})\neq\emptyset,$ we therefore deduced the following results from Theorem \ref{T1}:
  Suppose that
  \begin{enumerate}
  	\item [(C1)] $R_{\lambda,h}$ and  $R_{\lambda,l}$ be defined as above; 
  	\item [(C2)]  $\mathbb{D}_{j},j=1,2$ and $\mathbb{D}_{j}^{*},j=1,2$ as in Theorem \ref{1};
  	\item [(C3)] $(R_{\lambda,h}-I)$ and $(R_{\lambda,l}-I)$ are demiclosed at zero;
  	\item [(C4)] Let $\mathbb{U}$ and $\mathbb{V}$ be defined as 
  	\begin{equation}
  	{} \left\{\begin{array}{ll} \mathbb{U}=(1-\eta)I +\eta R_{\lambda,h}((1-\zeta)I+\zeta R_{\lambda,h}),
  	\\\mathbb{V}=(1-\eta)I +\eta R_{\lambda,l}((1-\zeta)I+\zeta R_{\lambda,l}), & \textrm{ $  $} \label{3} 
  	\end{array}\right.
  	\end{equation}
  	where $0<\eta<\zeta<\frac{1}{1+\sqrt{1+L^{2}}}.$\item [(C5)] \textbf{Algorithm:} Let $(x_{n},y_{n})\in \mathbb{H}_{1}\times \mathbb{H}_{2},$ be defined as 
  	\begin{equation}
  	{} \left\{\begin{array}{ll} x_{n+1}=(1-\alpha_{n}) v_{n}+\alpha_{n}\mathbb{U}v_{n},
  		\\v_{n}=(1-\tau_{n})x_{n}+\tau_{n}\mathbb{U} x_{n}+\tau_{n}\mathbb{D}^{*}_{1}(\mathbb{D}_{2}y_{n}-\mathbb{D}_{1}x_{n}), \forall n\geq 0;\\

  	\\ y_{n+1}=(1-\alpha_{n}) w_{n}+\alpha_{n}\mathbb{V}w_{n},
  	\\w_{n}=(1-\tau_{n}) y_{n}+\tau_{n}\mathbb{V} y_{n}+\tau_{n}\mathbb{D}^{*}_{2}(\mathbb{D}_{1}x_{n}-\mathbb{D}_{2}y_{n}), \forall n\geq 0;& \textrm{ $  $} \label{5} 
  	\end{array}\right.
  	\end{equation}

  	where    $x_{0}\in \mathbb{H}_{1}$ and $y_{0}\in \mathbb{H}_{2}$ are  chosen arbitrary,  $0<a<\alpha_{n}<1,$ and $\tau_{n}\in\left(0, \frac{2}{L_{1}+L_{2}}\right),$ where   $L_{1}=\mathbb{D}_{1}^{*}\mathbb{D}_{1}$ and $L_{2}=\mathbb{D}_{2}^{*}\mathbb{D}_{2},$ respectively.
   \end{enumerate}
  
  \begin{theorem}\label{Tgh} Suppose that  assumptions $(C1)-(C5)$ 
  	are satisfied and that $\Gamma_{3}\neq\emptyset.$  
  	Then, the sequence $\{(x_{n}, y_{n})\}$ 
  	generated by  (\ref{5}) 
  	converges weakly to the solution of SCMP.
  \end{theorem}\newpage
  \section{NUMERICAL EXAMPLES}
  This section gives  numerical examples that illustrated our results. 
  
  \begin{example}In Theorem (\ref{1}), Let $H_{1}=\mathbb R,$    $C,Q\in(0,\infty),$ and let $D_{1}x=\frac{x}{2}$, $D_{2}y=\frac{y}{3},$ then $D_{1}=D_{1}^{*} =\frac{1}{2}$ and $D_{2}=D_{2}^{*}=\frac{1}{3}, respectively$. Define $T:C\to \mathbb R$  and  $S:Q\to \mathbb R$ by 
  	$ Tx  = \frac{x+4}{2},\forall x \in C$ and $ Sy=\frac{y^{2}+2}{y+4}$, for all $y\in Q.$ Clearly, T and S are quasi pseudo demicontractive mappings with Fix(T)=4 and Fix(S)=$\frac{1}{2}$, and (I-T)  and (I-S) are demiclosed at zero.    In Algorithm \ref{1}, let $\eta=\frac{1}{2},$ $\xi=\frac{1}{5},$ $\tau_{n}=\frac{1}{7},$  $\alpha_{n}=\frac{1}{6},$ clearly, these parameters satisfies the hypothesis of Theorem \ref{1}. Setting the number of iteration n=100, by using maple, we obtain the following results: 
  	 	\end{example}

   \begin{table*}[hbt!]
  	\centering
  	\parbox{0.4\textwidth}{
  		\begin{footnotesize}
  			\begin{tabular}{c |c  |c}
  				\hline\\
  				n & $x_{n}$& $y_{n}$\\[0.5ex]
  				\hline\\
  			1 & 1.000000000&  1.000000000\\ 
  			2 & 1.446428571  & 0.8831268924\\
  			3 & 1.826424319  & 0.7916797163\\
  			. &.   &. \\
  			. &.  &. \\
  			73&.	&0.5000000002 \\
  			74 &.	&0.5000000000\\
  			. &.	& .\\
  			98 & 3.999999582  & 0.5000000000  \\
  			99 & 3.999999644	 &0.5000000000\\
  			100 & 3.999999697	& 0.5000000000    \\
  				\hline
  			\end{tabular}
  \caption{Numerical results of Algorithm \ref{1} starting with initial value $x_{1}=1$ and $y_{1}=1.$}
  		\label{tab:table}
  		\end{footnotesize}
		  	
  	}
  	\qquad
  	\begin{minipage}[c]{0.53\textwidth}%
  		\centering
  		\includegraphics[width=1\textwidth]{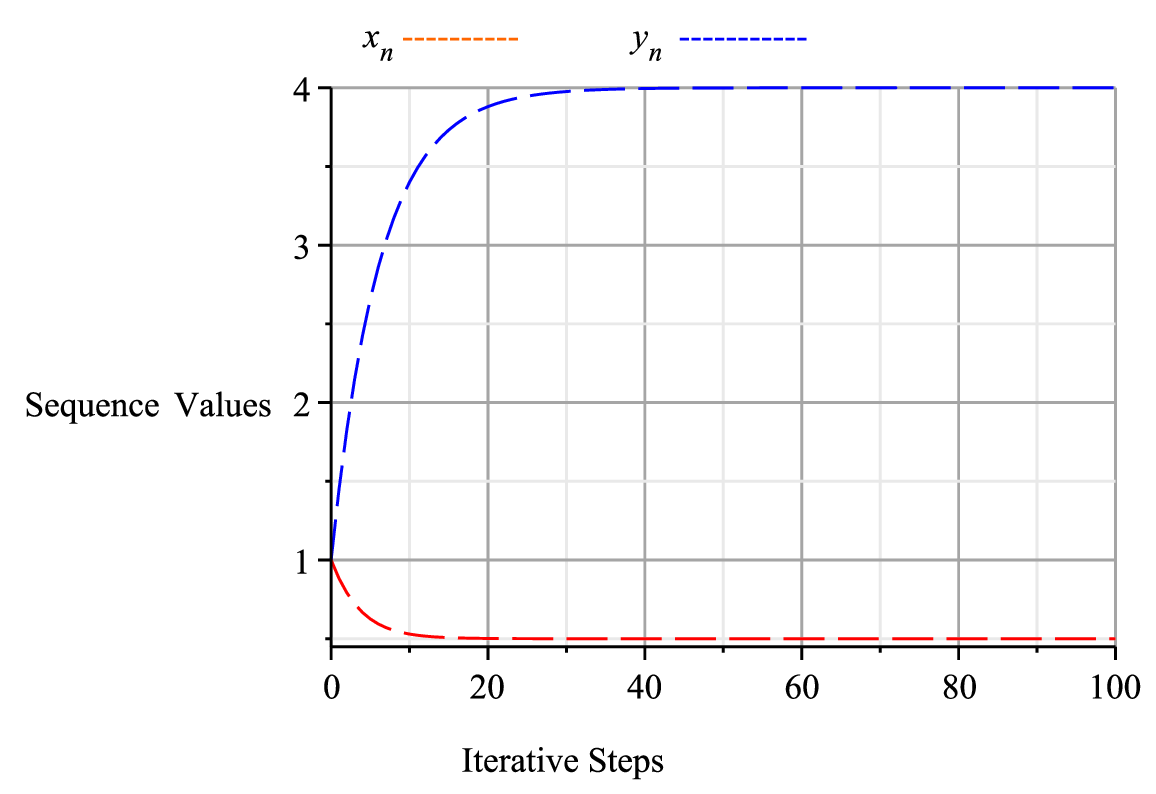}
  	Figure 1:	Graph of Algorithm \ref{1} at 100 iterations  starting with initial value $x_{1}=1$ and $y_{1}=1.$
  		\label{fig:figure}
  	\end{minipage}
  \end{table*}

\begin{example}In Theorem (\ref{1}), Let $H_{1}=\mathbb R,$    $C,Q\in(0,\infty).$ And let $Ax=\frac{x}{2}$, $By=\frac{y}{3},$ then $A=A^{*} =\frac{1}{2}$ and $B=B^{*}=\frac{1}{3}, respectively$. Define $T:C\to \mathbb R$  and  $S:Q\to \mathbb R$ by 
	$ Tx  = \frac{x^{5}+6}{x^{4}+2},\forall x \in C$ and $ Sy=\frac{y^{3}+4}{y^{2}+y}$, for all $y\in Q.$ Clearly, T and S are quasi pseudo demicontractive mappings with Fix(T)=2 and Fix(S)=2, and (I-T)  and (I-S) are demiclosed at zero.    In Algorithm \ref{1}, choose $\eta=\frac{1}{2},$ $\xi=\frac{1}{5},$ $\tau_{n}=\frac{1}{6},$  $\alpha_{n}=\frac{1}{7},$ clearly, these parameters satisfies the hypothesis of Theorem \ref{1}. Setting the number of iteration n=1000, by using maple, we obtain the following results: 
\end{example} 	
 \FloatBarrier
 \begin{table}[h]
 	\centering
 	\parbox{0.4\textwidth}{
 		\begin{footnotesize}
 				\begin{tabular}{c |c  |c}
 				\hline\\
 				n & $x_{n}$& $y_{n}$\\[0.5ex]
 				\hline\\
 				1 & 1.000000000&  1.000000000\\ 
 				2 & 1.402141502  &  1.283490816 \\
 				3 & 1.584779961  & 1.420942750\\
 				.&.   &. \\
 				. &.   &. \\
 				. &.  &. \\
 				80&1.999999998   &. \\
 				81&1.999999999	& . \\
 				. &.  &. \\
 				98 & 1.999999999  & 1.998915702   \\
 				99 &1.999999999 	 &1.998975310 \\
 				100 & 1.999999999	&1.999031634   \\
 				\hline
 			\end{tabular}
 		\end{footnotesize}
 		\caption{Numerical results of Algorithm \ref{1} starting with initial value $x_{1}=1$ and $y_{1}=1.$}
 		\label{tab:table}
 	}
 	\qquad
 	\begin{minipage}[c]{0.53\textwidth}
 		\centering
 		\includegraphics[width=1\textwidth]{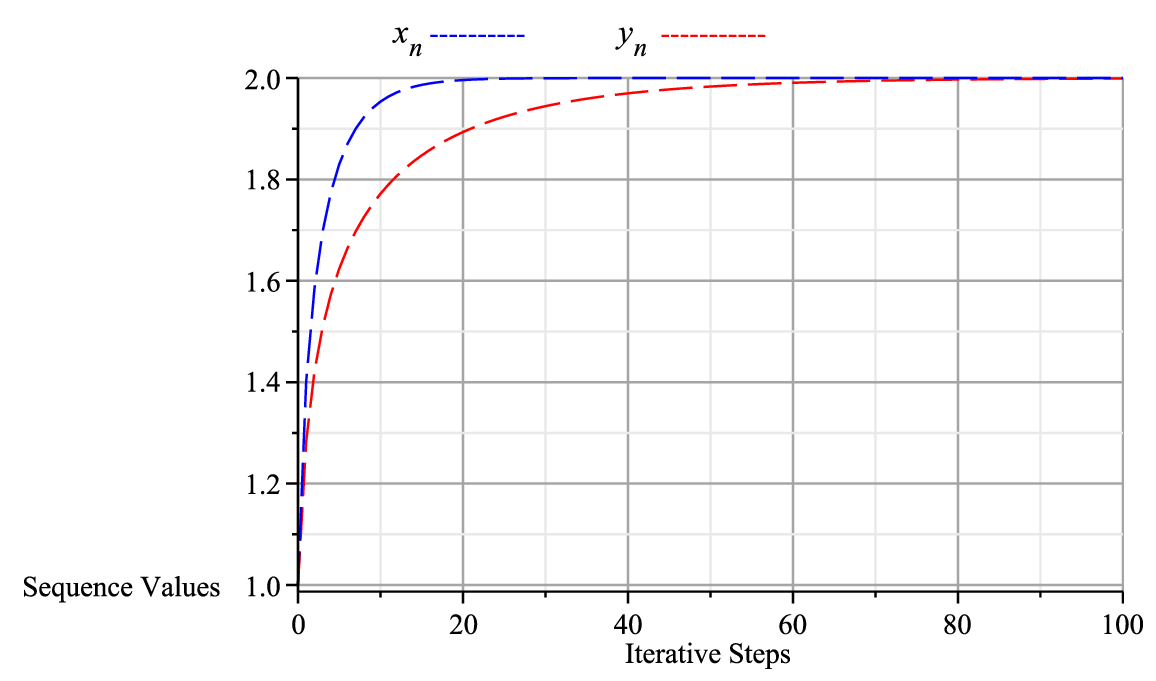}
 		Figure 2: Numerical results of Algorithm \ref{1} starting with initial value $x_{1}=1$ and $y_{1}=1.$
 		\label{fig:figure}
 	\end{minipage}
 \end{table}
 \FloatBarrier

  \section{Conclusion}	
It is generally known that in order to solve the split equality fixed-point problem (SEFPP), it is necessary to compute the norm of bounded and linear operators, which is a challenging task in real life.To address this issue, we studied the SEFPP involving the class of quasi-pseudocontractive mappings in Hilbert spaces and constructed novel algorithms in this regard, and we proved the algorithms' convergence both with and without prior knowledge of the operator norm for bounded and linear mappings. Additionally, we gave applications and numerical examples of our findings. The discoveries highlighted in this work contribute to the generalization of various well-known findings documented in the literature. Furthermore, quasi-pseudocontractive mappings encompass various types, including quasi-nonexpansive, demicontractive, and directed mappings.

The SEFPP explored in our study is highly general. As special examples, it covers a wide range of problems, including split fixed points, split equality, and split feasibility problems. Our findings not only complement and generalize the findings  in Chang et al., \cite{a}, Moudafi \cite{mou,moudafi}  and Chang et al. \cite{2}, but  also offer a cohesive framework for researching further problems pertaining to the SEFPP.

Finally, strong convergence was obtained by imposing the semi-compactness condition. This compactness  condition appears very strong as some mappings are not semi-compact; therefore, new researches can be carried out to prove the strong convergence results without imposing the compactness condition,

\end{document}